\documentclass{amsart}

\usepackage{amsmath,amssymb}
\usepackage{amsthm, amsrefs}
\usepackage{latexsym}
\usepackage{indentfirst}
\usepackage{mathrsfs}
\usepackage{placeins}
\usepackage{graphicx}

\numberwithin{equation}{section}

\theoremstyle{plain}
\newtheorem{theorem}{Theorem}[section]
\newtheorem{example}{Example}[section]
\newtheorem{lemma}[theorem]{Lemma}

\theoremstyle{definition}

\theoremstyle{remark}
\newtheorem*{remark}{Remark}

\DeclareMathOperator{\tr}{tr}

\newcommand{\ie}{\textit{i.e.}}

\newcommand{\ud}{\,\mathrm{d}}
\newcommand{\RR}{\mathbb{R}}

\newcommand{\Or}{\mathcal{O}}

\newcommand{\bd}[1]{\boldsymbol{#1}}

\newcommand{\veps}{\varepsilon}

\newcommand{\abs}[1]{\lvert#1\rvert}
\newcommand{\norm}[1]{\lVert#1\rVert}

\newcommand{\I}{\imath}

\renewcommand{\Im}{\mathfrak{Im}}

\newcommand{\barint}{\kern4pt \raise3.4pt\hbox{\vrule height.6pt
    width7pt} \kern-11pt \int}

\newcommand{\FGA}{\mathrm{FGA}}
\newcommand{\GBM}{\mathrm{GBM}}

\begin{document}

\title[Frozen Gaussian approximation]{Frozen Gaussian
  approximation for high frequency wave propagation}

\author{Jianfeng Lu}
\address{Department of Mathematics \\
  Courant Institute of Mathematical Sciences \\
  New York University \\
  New York, NY 10012 \\
  email: jianfeng@cims.nyu.edu }

\author{Xu Yang}
\address{Program in Applied and Computational Mathematics \\
  Princeton University \\
  Princeton, NJ 08544 \\
  email: xuyang@math.princeton.edu }

\date{August 8, 2010. {\it Revised} : November 11, 2010}

\thanks{Both authors are grateful to Weinan E for his inspiration and
  helpful suggestions and discussions that improve the
  presentation. J.L.~would like to thank Lexing Ying for stimulating
  discussions and for providing preprints \cites{QiYi:app1, QiYi:app2}
  before publication.  Part of the work was done during J.L.'s visits
  to Institute of Computational Mathematics and Scientific/Engineering
  Computing of Chinese Academy of Sciences and Peking University, and
  X.Y.'s visits to Tsinghua University and Peking University. We
  appreciate their hospitality. We thank anonymous referees for their
  valuable suggestions and remarks. X.Y. was partially supported by
  the DOE grant DE-FG02-03ER25587, the NSF grant DMS-0708026 and the
  AFOSR grant FA9550-08-1-0433.}

\begin{abstract}
  We propose the frozen Gaussian approximation for computation of high
  frequency wave propagation. This method approximates the solution to
  the wave equation by an integral representation.  It provides a
  highly efficient computational tool based on the asymptotic analysis
  on phase plane. Compared to geometric optics, it provides a valid
  solution around caustics. Compared to the Gaussian beam method, it
  overcomes the drawback of beam spreading. We give several numerical
  examples to verify that the frozen Gaussian approximation performs
  well in the presence of caustics and when the Gaussian beam
  spreads. Moreover, it is observed numerically that the frozen Gaussian 
  approximation exhibits better accuracy than the Gaussian beam method.
\end{abstract}

\maketitle

\section{Introduction}

We are interested in developing efficient numerical methods for high
frequency wave propagation. For simplicity and clarity we take the
following linear scalar wave equation to present the idea,
\begin{equation}\label{eq:wave}
  \partial_t^2 u - c^2(\bd{x}) \Delta u = 0,\qquad \bd{x}\in\RR^d,
\end{equation}
with WKB initial conditions,
\begin{equation}\label{eq:WKBini}
  \begin{cases}
    u_0(\bd{x}) = A_0(\bd{x}) e^{\frac{\I}{\veps} S_0(\bd{x})}, \\
    \partial_t u_0(\bd{x}) = \frac{1}{\veps} B_0(\bd{x})
    e^{\frac{\I}{\veps} S_0(\bd{x})},
  \end{cases}
\end{equation}
where $u$ is the wave field, $d$ is the dimensionality and
$\I=\sqrt{-1}$ is the imaginary unit. We assume that the local wave
speed $c(\bd{x})$ is a smooth function. The small parameter $\veps \ll
1$ characterizes the high frequency nature of the wave. The proposed
method can be generalized to other types of wave equations
\cite{LuYang:MMS}.

Numerical computation of high frequency wave propagation is an
important problem arising in many applications, such as
electromagnetic radiation and scattering, seismic and acoustic waves
traveling, just to name a few. It is a two-scale problem. The {\it
  large} length scale comes from the characteristic size of the
medium, while the {\it small} length scale is the wavelength. The
disparity between the two length scales makes direct numerical
computations extremely hard. In order to achieve accurate results, the
mesh size has to be chosen comparable to the wavelength or even
smaller. On the other hand, the domain size is large so that a huge
number of grid points are needed.

In order to compute efficiently high frequency wave propagation,
algorithms based on asymptotic analysis have been developed. One of
the most famous examples is geometric optics. In the method, it is
assumed that the solution has a form of
\begin{equation}\label{eq:GO}
  u(t,\bd{x})=A(t,\bd{x})e^{\I S(t,\bd{x})/\veps}.
\end{equation}
To the leading order, the phase function $S(t,\bd{x})$ satisfies the
eikonal equation,
\begin{equation}\label{eq:eikonal}
  \abs{\partial_tS}^2-c^2(\bd{x})
  \abs{\nabla_{\bd{x}}S}^2=0,
\end{equation}
and the amplitude $A(t,\bd{x})$ satisfies the transport equation,
\begin{equation*}
\partial_tA-c^2(\bd{x})\frac{\nabla_{\bd{x}}S}{\partial_t
S}\cdot\nabla_{\bd{x}}A+\frac{\bigl(\partial_t^2S-c^2(\bd{x})\Delta
S\bigr)}{2\partial_t S}A=0.
\end{equation*}
The merit of geometric optics is that it only solves the macroscopic
quantities $S(t,\bd{x})$ and $A(t,\bd{x})$ which are
$\veps$-independent. Computational methods based on the geometric
optics are reviewed in \cites{EnRu:03, Ru:07}.

However, since the eikonal equation \eqref{eq:eikonal} is of
Hamilton-Jacobi type, the solution of \eqref{eq:eikonal} becomes
singular after the formation of caustics. At caustics, the
approximate solution of geometric optics is invalid since the amplitude
$A(t,\bd{x})$ blows up. To overcome this problem, Popov introduced
Gaussian beam method in \cite{Po:82}. The single beam solution of
the Gaussian beam method has a similar form to geometric optics,
\[ u(t,\bd{x})=A(t,\bd{y})e^{\I \tilde{S}(t,\bd{x},\bd{y})/\veps}.\]
The difference lies in that the Gaussian beam method uses a {\it
complex} phase function,
\begin{equation}\label{eq:GBphs}
  \tilde{S}(t,\bd{x},\bd{y})=S(t,\bd{y})+\bd{p}(t,\bd{y})\cdot(\bd{x}-\bd{y})+
  \frac{1}{2}(\bd{x}-\bd{y})\cdot
  M(t,\bd{y})(\bd{x}-\bd{y}),
\end{equation}
where $S\in\RR,\;\bd{p}\in\RR^d,\;M\in\mathbb{C}^{d\times d}$. The
imaginary part of $M$ is chosen to be positive definite so that the
solution decays exponentially away from $\bd{x}=\bd{y}$, where
$\bd{y}$ is called the beam center. This makes the solution a Gaussian
function, and hence the method was named the Gaussian beam method.  If
the initial wave is not in a form of single beam, one can approximate
it by using a number of Gaussian beams. The validity of this
construction at caustics was analyzed by Ralston in \cite{Ra:82}.
Recently, there have been a series of numerical studies including both
the Lagrangian type \cites{TaQiRa:07, Ta:08, MoRu:app, QiYi:app1,
  QiYi:app2} and the Eulerian type \cites{LeQiBu:07, LeQi:09,
  JiWuYa:08, JiWuYaHu:10, JiWuYa:10, JiWuYa:11, LiRa:09, LiRa:10}.

The construction of Gaussian beam approximation is based on the
truncation of the Taylor expansion of $\tilde{S}$ around the beam
center $\bd{y}$ up to the quadratic term, hence it loses accuracy when
the width of the beam becomes large, \ie, when the imaginary part of
$M(t, \bd{y})$ in \eqref{eq:GBphs} becomes small so that the Gaussian
function is not localized any more. This happens for example when the
solution of the wave equation spreads (the opposite situation of
forming caustics). This is a severe problem in general, as shown by
examples in Section \ref{sec:numer}. One could overcome the problem of
spreading of beams by doing reinitialization once in a while, see
\cites{QiYi:app1, QiYi:app2}.  This increases the computational
complexity especially when beams spread quickly.

Therefore a method working in both scenario of spreading and caustics
is required. The main idea of the method proposed in the current work
is to use Gaussian functions with fixed widths, instead of using those
that might spread over time, to approximate the wave solution. That is
why this type of method is called frozen Gaussian approximation (FGA).
Despite its superficial similarity with the Gaussian beam method
(GBM), it is different at a fundamental level.  FGA is based on phase
plane analysis, while GBM is based on the asymptotic solution to a
wave equation with Gaussian initial data.  In FGA, the solution to the
wave equation is approximated by a superposition of Gaussian functions
living in the phase space, and each function is {\it not} necessarily
an asymptotic solution, while GBM uses Gaussian functions (named as
beams) in the physical space, with each individual beam being an
asymptotic solution to the wave equation. The main advantage of FGA
over GBM is that the problem of beam spreading no longer
exists.\footnote{Divergence is still an issue for the Lagrangian
  approach, one needs to work in the Eulerian framework to completely
  solve the problem, which is considered in \cite{LuYang:MMS}.}
Besides, numerically we observe that FGA has better accuracy than
GBM when keeping the same order of terms in asymptotic series.
% \footnote{This will be further discussed in
%   Section~\ref{sec:numer}.}.
On the other hand, the solution given by FGA is asymptotically
accurate around caustics where geometric optics breaks down.

Our work is motivated by the chemistry literature on the propagation
of time dependent Schr\"odinger equation, where the spreading of
solution is a common phenomenon, for example, in the dynamics of a
free electron. In \cite{He:81}, Heller introduced frozen Gaussian
wavepackets to deal with this issue, but it only worked for a short
time propagation of order $\Or(\hbar)$ where $\hbar$ is the Planck
constant. To make it valid for longer time of order $\Or(1)$, Herman
and Kluk proposed in \cite{HeKl:84} to change the weight of Gaussian
packets by adding so-called Herman-Kluk prefactor. Integral
representation and higher order approximations were developed by Kay
in \cite{Ka:94} and \cite{Ka:06}. Recently, the semiclassical
approximation underlying the method was analyzed rigorously by Swart
and Rousse in \cite{SwRo:09} and also Robert in \cite{Ro:09}. We
generalize their ideas for propagation of high frequency waves,
aiming at developing an efficient computational method. We decompose
waves into several branches of propagation, and each of them is
approximated using Gaussian functions on phase plane. Their
centers follow different Hamiltonian dynamics for different
branches. Their weight functions, which are analogous to the
Herman-Kluk prefactor, satisfy new evolution equations derived from
asymptotic analysis.

The rest of paper is organized as follows. In Section \ref{sec:formu},
we state the formulations and numerical algorithm of the frozen
Gaussian approximation. In Section \ref{sec:asymder}, we provide
asymptotic analysis to justify the formulations introduced in Section
\ref{sec:formu}. The numerical examples are given in Section
\ref{sec:numer} to verify the accuracy and to compare the frozen
Gaussian approximation (FGA) with the Gaussian beam method (GBM). In
Section \ref{sec:conclusion}, we discuss the efficiency of FGA in comparison
with GBM and higher order GBM, with some comments on the phenomenon of 
error cancellation, and we give some conclusive remarks in the end.

\section{Formulation and algorithm}\label{sec:formu}

In this section we present the basic formulation and the main
algorithm of the frozen Gaussian approximation (FGA), and leave the
derivation to the next section.

\subsection{Formulation}

FGA approximates the solution to the wave equation \eqref{eq:wave}
by the integral representation,
\begin{equation}\label{eq:ansatz}
  \begin{aligned}
    u^{\FGA}(t, \bd{x}) & = \frac{1}{(2\pi \veps)^{3d/2}} \int_{\RR^{3d}}
    a_+(t, \bd{q}, \bd{p}) e^{\frac{\I}{\veps} \Phi_+(t, \bd{x},
      \bd{y}, \bd{q}, \bd{p})} u_{+,0}(\bd{y}) \ud \bd{y} \ud \bd{p}
    \ud \bd{q} \\
    & + \frac{1}{(2\pi \veps)^{3d/2}} \int_{\RR^{3d}} a_-(t, \bd{q},
    \bd{p}) e^{\frac{\I}{\veps} \Phi_-(t, \bd{x}, \bd{y}, \bd{q} ,
      \bd{p})} u_{-,0}(\bd{y}) \ud \bd{y} \ud \bd{p} \ud \bd{q},
  \end{aligned}
\end{equation}
where $u_{\pm, 0}$ are determined by the initial value,
\begin{equation}\label{eq:ini_2branch}
  u_{\pm,0}(\bd{x}) = A_{\pm}(\bd{x}) e^{\frac{\I}{\veps} S_0(\bd{x})},
\end{equation}
with
\begin{equation*}
  A_{\pm}(\bd{x}) = \frac{1}{2} \biggl( A_0(\bd{x}) \pm
  \frac{\I B_0(\bd{x})}{ c(\bd{x}) \abs{ \partial_{\bd{x}} S_0(\bd{x})}}
  \biggr).
\end{equation*}
The equation \eqref{eq:ansatz} implies that the solution consists of
two branches (``$\pm$'').

In \eqref{eq:ansatz}, $\Phi_{\pm}$ are the phase functions given by
\begin{multline}\label{eq:phi}
  \Phi_{\pm}(t, \bd{x}, \bd{y}, \bd{q}, \bd{p}) = \bd{P}_{\pm}(t,
  \bd{q}, \bd{p}) \cdot ( \bd{x} -
  \bd{Q}_{\pm}(t, \bd{q}, \bd{p})) - \bd{p} \cdot ( \bd{y} - \bd{q} ) \\
  + \frac{\I}{2} \abs{\bd{x} - \bd{Q}_{\pm}(t, \bd{q}, \bd{p})}^2 +
  \frac{\I}{2} \abs{\bd{y} - \bd{q}}^2.
\end{multline}
Given $\bd{q}$ and $\bd{p}$ as parameters, the evolution of
$\bd{Q}_{\pm}$ and $\bd{P}_{\pm}$ are given by the equation of
motion corresponding to the Hamiltonian $H_{\pm} = \pm
c(\bd{Q}_{\pm}) \abs{\bd{P}_{\pm}}$,
\begin{equation}\label{eq:characline}
  \begin{cases}
    \displaystyle
    \frac{\ud \bd{Q}_{\pm}}{\ud t}  = \partial_{\bd{P}_{\pm}} H_{\pm} = \pm c
    \frac{\bd{P}_{\pm}}{\abs{\bd{P}_{\pm}}},\\
    \displaystyle
    \frac{\ud \bd{P}_{\pm}}{\ud t} = - \partial_{\bd{Q}_{\pm}} H_{\pm}
    = \mp \partial_{\bd{Q}_{\pm}} c \abs{\bd{P}_{\pm}},
  \end{cases}
\end{equation}
with the initial conditions $\bd{Q}_{\pm}(0, \bd{q}, \bd{p}) =
\bd{q}$ and $\bd{P}_{\pm}(0, \bd{q}, \bd{p}) = \bd{p}$. The
evolution equation of $a_{\pm}$ is given by
\begin{equation}\label{eq:partialta}
\begin{aligned}
  \frac{\ud a_{\pm}}{\ud t} &= \pm\frac{a_{\pm}}{2} \Bigl(
  \frac{\bd{P}_{\pm}}{\abs{\bd{P}_{\pm}}}\cdot
  \partial_{\bd{Q}_{\pm}} c - \frac{(d-1)\I}{\abs{\bd{P}_{\pm}}} c \Bigr) \\
  &\pm \frac{a_{\pm}}{2} \tr\biggl( Z_{\pm}^{-1} \partial_{\bd{z}} \bd{Q}_{\pm} \Bigl( 2
  \frac{\bd{P}_{\pm}}{\abs{\bd{P}_{\pm}}} \otimes \partial_{\bd{Q}_{\pm}} c \\&\hspace{6em} - \frac{\I
    c}{\abs{\bd{P}_{\pm}}} \Bigl( \frac{\bd{P}_{\pm}\otimes
    \bd{P}_{\pm}}{\abs{\bd{P}_{\pm}}^2} - I \Bigr) - \I
  \abs{\bd{P}_{\pm}}\partial_{\bd{Q}_{\pm}}^2 c \Bigr) \biggr)
\end{aligned}
\end{equation}
with the initial condition,
\begin{equation*}
  a_{\pm}(0, \bd{q}, \bd{p}) = 2^{d/2}.
\end{equation*}
In \eqref{eq:partialta}, $\bd{P}_{\pm}$ and $\bd{Q}_{\pm}$ are
evaluated at $(t, \bd{q}, \bd{p})$, $c$ and $\partial_{\bd{Q}_{\pm}}
c$ are evaluated at $\bd{Q}_{\pm}$, $I$ is the identity matrix, and
we have introduced short hand notations
\begin{equation}\label{eq:op_zZ}
  \partial_{\bd{z}}=\partial_{\bd{q}}-\I\partial_{\bd{p}},
  \qquad
  Z_{\pm}=\partial_{\bd{z}}(\bd{Q}_{\pm}+\I\bd{P}_{\pm}).
\end{equation}
The evolution of the weight $a_{\pm}$ is analogous to the
Herman-Kluk prefactor \cite{HeKl:84}.

\begin{remark}
    1. The equation \eqref{eq:partialta} can be reformulated as
    \begin{equation}\label{eq:a_formu}
    \frac{\ud a_{\pm}}{\ud t}=\pm a_{\pm}\frac{\bd{P}_{\pm}}{\abs{\bd{P}_{\pm}}}\cdot
    \partial_{\bd{Q}_{\pm}} c+\frac{a_{\pm}}{2}
    \tr\left(Z_{\pm}^{-1}\frac{\ud Z_{\pm}}{\ud t} \right).
    \end{equation}
    When $c$ is constant, \eqref{eq:a_formu} has an analytical
    solution $a_{\pm}=(\det Z_{\pm})^{1/2}$ with the branch of square root
    determined continuously in time by the initial value.

    2. $\partial_{\bd{z}}\bd{Q}_{\pm}$ and $\partial_{\bd{z}}\bd{P}_{\pm}$ satisfy the
    following evolution equations
    \begin{align}\label{eq:dzQ}
      &\frac{\ud (\partial_{\bd{z}}\bd{Q}_{\pm})}{\ud
        t}=\pm\partial_{\bd{z}}\bd{Q}_{\pm}
      \frac{\partial_{\bd{Q}_{\pm}}c\otimes\bd{P}_{\pm}}{\abs{\bd{P}_{\pm}}}\pm
      c\partial_{\bd{z}}\bd{P}_{\pm}\left(\frac{I}{\abs{\bd{P}_{\pm}}}-
        \frac{\bd{P}_{\pm}\otimes\bd{P}_{\pm}}{\abs{\bd{P}_{\pm}}^3}
      \right), \\
      &\frac{\ud (\partial_{\bd{z}}\bd{P}_{\pm})}{\ud
        t}=\mp\partial_{\bd{z}}\bd{Q}_{\pm}\partial_{\bd{Q}_{\pm}}^2c\abs{\bd{P}_{\pm}}
      \mp\partial_{\bd{z}}\bd{P}_{\pm}
      \frac{\bd{P}_{\pm}\otimes\partial_{\bd{Q}_{\pm}}c}{\abs{\bd{P}_{\pm}}}.\label{eq:dzP}
    \end{align}
    One can solve \eqref{eq:dzQ}-\eqref{eq:dzP} to get
    $\partial_{\bd{z}}\bd{Q}_{\pm}$ and $\partial_{\bd{z}}\bd{P}_{\pm}$ in
    \eqref{eq:partialta}. This increases the computational cost, but
    avoids the errors of using divided difference to approximate
    derivative.

\end{remark}

Notice that \eqref{eq:ansatz} can be rewritten as
\begin{equation}\label{eq:recons}
  \begin{aligned}
    u^{\FGA}(t, \bd{x}) & = \int_{\RR^{2d}} \frac{a_+}{(2\pi \veps)^{3d/2}} \psi_+
    e^{\frac{\I}{\veps}\bd{P}_+\cdot(\bd{x} - \bd{Q}_+) -
      \frac{1}{2\veps} \abs{\bd{x}
        - \bd{Q}_+}^2} \ud \bd{p} \ud \bd{q} \\
    & + \int_{\RR^{2d}} \frac{a_-}{(2\pi
      \veps)^{3d/2}} \psi_-
      e^{\frac{\I}{\veps}
      \bd{P}_-\cdot(\bd{x} - \bd{Q}_-) - \frac{1}{2\veps} \abs{\bd{x}
        - \bd{Q}_-}^2} \ud \bd{p} \ud \bd{q},
  \end{aligned}
\end{equation}
where
\begin{equation}\label{eq:psi}
  \psi_{\pm}(\bd{q}, \bd{p}) =
   \int_{\RR^d} u_{\pm, 0}(\bd{y})
  e^{- \frac{\I}{\veps} \bd{p}\cdot(\bd{y} - \bd{q}) -
    \frac{1}{2\veps} \abs{\bd{y} - \bd{q}}^2} \ud\bd{y}.
\end{equation}
Therefore, the method first decomposes the initial wave into several
Gaussian functions in phase space, and then propagate the center of
each function along the characteristic lines while keeping the width
of the Gaussian fixed. This vividly explains the name frozen
Gaussian approximation of this method.

The formulation above gives the leading order frozen Gaussian
approximation with an error of $\Or(\veps)$. It is not hard to
obtain higher order approximations by the asymptotics presented in
Section \ref{sec:asymder}. We will focus mainly on the leading order
approximation in this paper and leave the higher order corrections
and rigorous numerical analysis to future works.

\subsection{Algorithm}

We first give a description of the overall algorithm. To construct the
frozen Gaussian approximation on a mesh of $\bd{x}$, one needs to
compute the integral \eqref{eq:recons} numerically with a mesh of
$(\bd{q},\bd{p})$. This will relate to the numerical computation of
\eqref{eq:psi} with a mesh of $\bd{y}$. Hence three different meshes
are needed in the algorithm. Moreover, the stationary phase
approximation implies that $\psi_{\pm}$ in \eqref{eq:psi} is localized
around the submanifold $\bd{p}=\nabla_{\bd{q}}S_0(\bd{q})$ on phase
plane for WKB initial conditions \eqref{eq:WKBini} when $\veps$ is
small. This means we only need to put the mesh grids of $\bd{p}$
around $\nabla_{\bd{q}}S_0(\bd{q})$ initially to get a good
approximation of the initial value. A one-dimensional example is given
to illustrate this localization property of $\psi_{\pm}$ in Figure
\ref{fig:meshpq} (left). The associated mesh grids are shown in Figure
\ref{fig:meshpq} (right).

\begin{figure}[htp]
\begin{tabular}{cc}
\resizebox{2.3in}{!}{\includegraphics{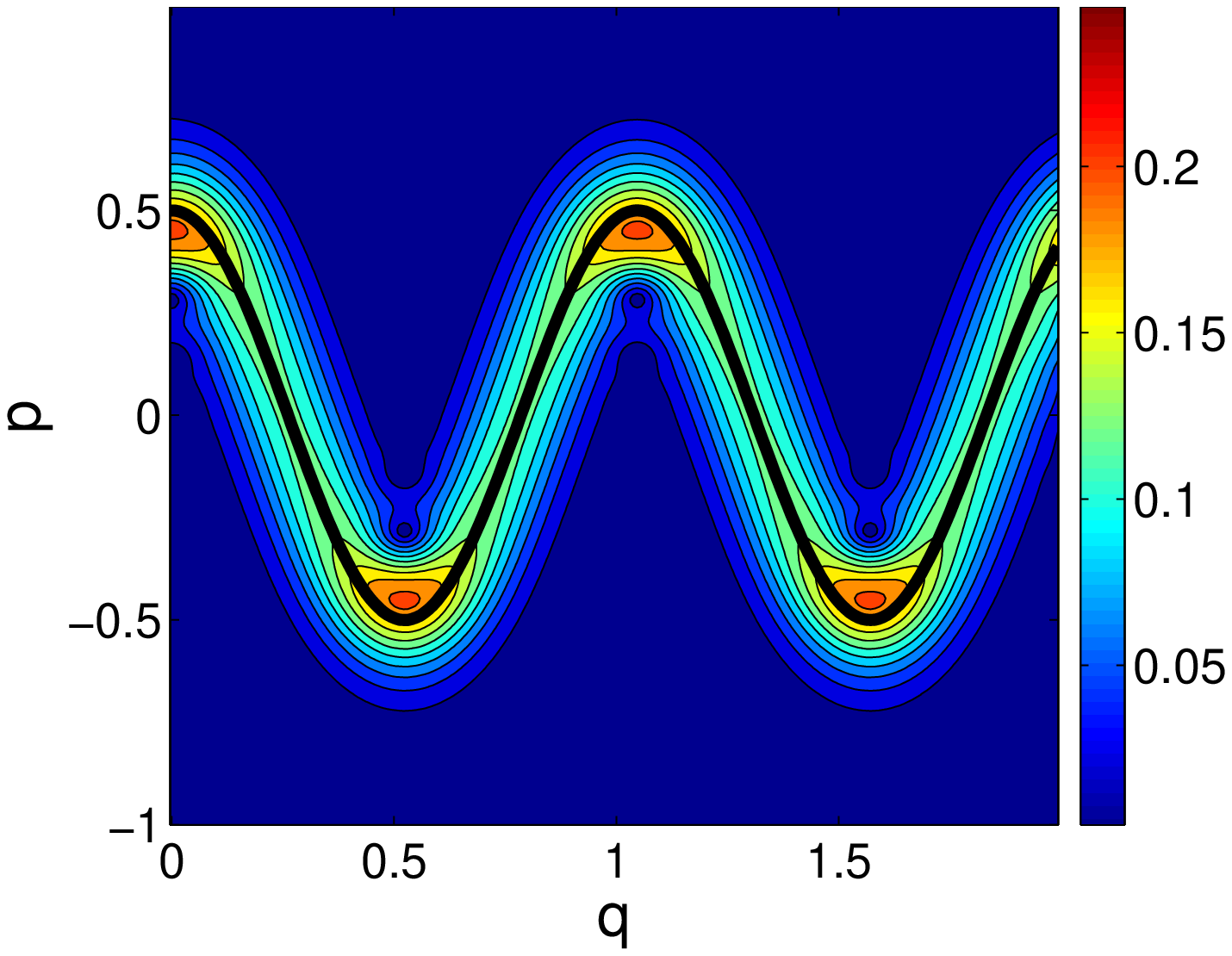}} &
\resizebox{2.3in}{!}{\includegraphics{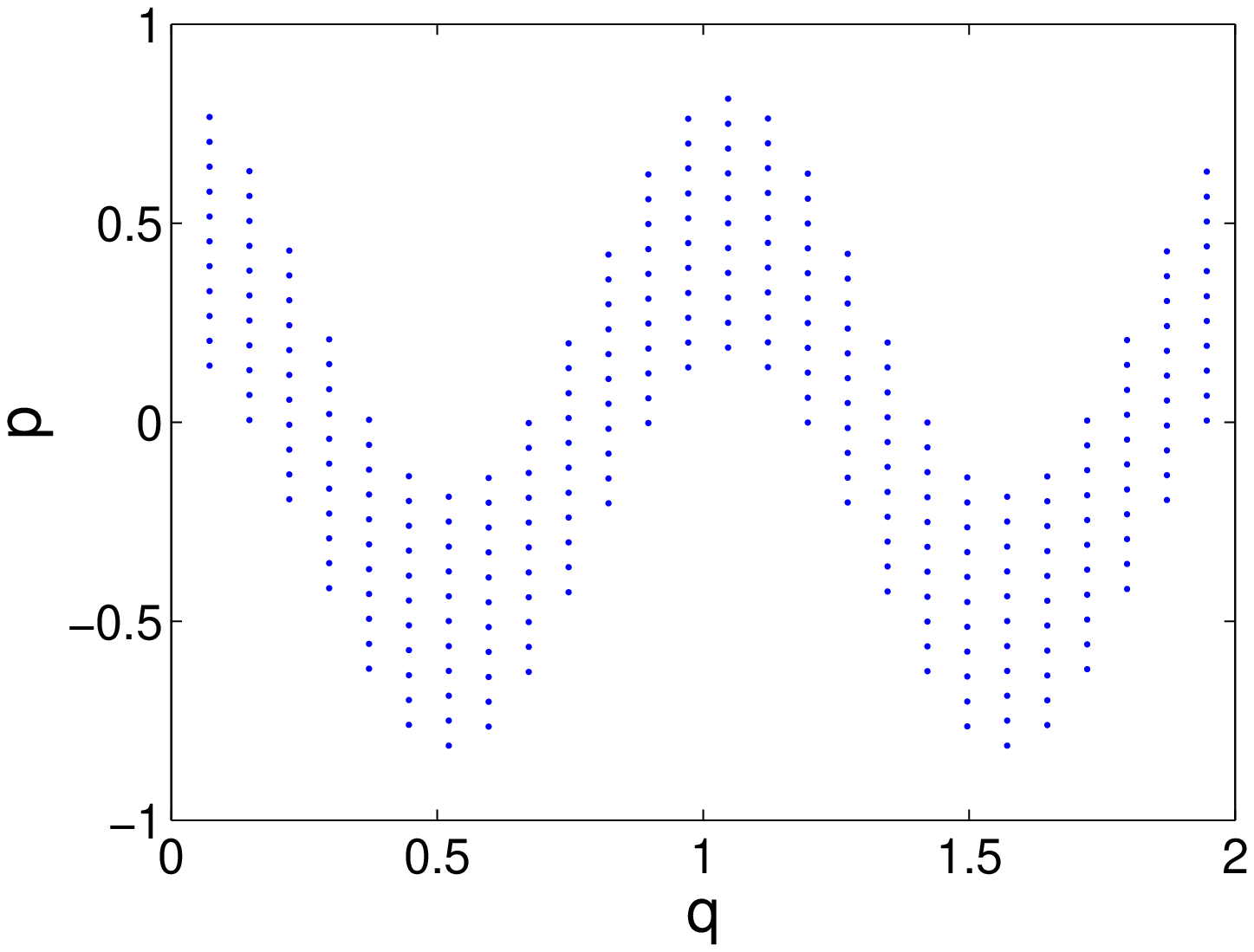}}
\end{tabular}
\caption{Left: an illustration of the localization of $\psi_+$ on
$(q,p)$ domain for
$u_{+,0}(y)=\exp\left(\I\frac{\sin(6y)}{12\veps}\right)$,
$\veps=1/128$; the black solid curve is $p=\cos(6q)/2$. Right: the
corresponding mesh grids of $(q,p)$.}\label{fig:meshpq}
\end{figure}

Next we describe in details all the meshes used in the algorithm.
\begin{enumerate}
\item Discrete mesh of $(\bd{q},\bd{p})$ for initializing $\bd{Q},\bd{P}$.
Denote $\bd{\delta q}=(\delta q_1,\cdots,\delta q_d)$ and
$\bd{\delta p}=(\delta p_1,\cdots,\delta p_d)$ as the mesh size.
Suppose $\bd{q}^0=(q^0_1,\cdots,q^0_d)$ is the starting point, then
the mesh grids $\bd{q^k}$, $\bd{k}=(k_1,\cdots,k_d)$, are defined as
\[\bd{q^k}=\bigl(q^0_1+(k_1-1)\delta q_1,\cdots, q^0_d+(k_d-1)\delta q_d\bigr),\]
where $k_j=1,\cdots,N_q$ for each $j\in\{1,\cdots,d\}$.

The mesh grids $\bd{p^{k,\ell}}$,
$\bd{\ell}=(\ell_1,\cdots,\ell_d)$, are defined associated with the
mesh grids $\bd{q^k}$,
\[\bd{p^{k,\ell}}=\bigl(\partial_{q_1}S_0(\bd{q^k})+\ell_1\delta
p_1,\cdots,
\partial_{q_d}S_0(\bd{q^k})+\ell_d \delta p_d\bigr),\] where
$\ell_j=-N_p,\cdots,N_p$ for each $j\in\{1,\cdots,d\}$.

\item Discrete mesh of $\bd{y}$ for evaluating $\psi_{\pm}$ in
\eqref{eq:psi}. $\bd{\delta y}=(\delta y_1,\cdots,\delta y_d)$ is
the mesh size. Denote $\bd{y}^0=(y^0_1,\cdots,y^0_d)$ as the
starting point. The mesh grids $\bd{y^m}$ are,
$\bd{m}=(m_1,\cdots,m_d)$,
\[\bd{y^m}=\bigl(y^0_1+(m_1-1)\delta y_1,\cdots, y^0_d+(m_d-1)\delta y_d\bigr),\]
where $m_j=1,\cdots,N_y$ for each $j\in\{1,\cdots,d\}$.

\item Discrete mesh of $\bd{x}$ for reconstructing the final solution.
$\bd{\delta x}=(\delta x_1,\cdots,\delta x_d)$ is the mesh size.
Denote $\bd{x}^0=(x^0_1,\cdots,x^0_d)$ as the starting point. The
mesh grids $\bd{x^n}$ are, $\bd{n}=(n_1,\cdots,n_d)$,
\[\bd{x^n}=\bigl(x^0_1+(n_1-1)\delta x_1,\cdots, x^0_d+(n_d-1)\delta x_d\bigr),\]
where $n_j=1,\cdots,N_x$ for each $j\in\{1,\cdots,d\}$.

\end{enumerate}

With the preparation of the meshes, we introduce the algorithm as
follows.

\begin{itemize}
\item[Step 1.] Decompose the initial conditions \eqref{eq:WKBini} into two
  branches of waves according to \eqref{eq:ini_2branch}.

\item[Step 2.] Compute the weight function $\psi_{\pm}$ by
  \eqref{eq:psi} for $(\bd{Q},\bd{P})$ initialized at $(\bd{q^{k}},\bd{p^{k,\ell}})$,
  \begin{multline}\label{eq:discpsi}
    \psi_{\pm}(\bd{q^{k}},\bd{p^{k,\ell}})=\sum_{\bd{m}}
    e^{\frac{\I}{\veps}(- \bd{p^{k,\ell}} \cdot ( \bd{y^m} - \bd{q^{k}}
      ) +\frac{\I}{2} \abs{\bd{y^{m}} - \bd{q^{k}}}^2)}\\\times
    u_{\pm,0}(\bd{y^m}) r_{\theta}(\abs{\bd{y^{m}} -
      \bd{q^{k}}})\delta y_1\cdots \delta y_d,
  \end{multline}
  where $r_{\theta}$ is a cutoff function such that $r_{\theta} = 1$
  in the ball of radius $\theta>0$ centered at origin and $r_{\theta}
  = 0$ outside the ball.

\item[Step 3.] Solve \eqref{eq:characline}-\eqref{eq:partialta} with
  the initial conditions
  \begin{align*}
    & \bd{Q}_{\pm}(0,\bd{q^{k}},\bd{p^{k,\ell}}) = \bd{q^{k}},\qquad
    \bd{P}_{\pm}(0,\bd{q^{k}},\bd{p^{k,\ell}}) =
    \bd{p^{k,\ell}},\\
    & a_{\pm}(0,\bd{q^{k}},\bd{p^{k,\ell}})=2^{d/2},
  \end{align*}
  by standard numerical integrator for ODE, for example the
  fourth-order Runge-Kutta scheme.  Denote the numerical solutions as
  $(\bd{Q}_{\pm}^{\bd{k,\ell}},\bd{P}_{\pm}^{\bd{k,\ell}})$ and
  $a_{\pm}^{\bd{k,\ell}}$.

\item[Step 4.] Reconstruct the solution by \eqref{eq:recons},
\begin{equation}\label{eq:discFGB}
  \begin{aligned}
    u^{\FGA}(t, \bd{x^{n}}) & = \sum_{\bd{k,\ell}} \biggl(
    \frac{a_+^{\bd{k,\ell}}r_{\theta}^+} {(2\pi \veps)^{3d/2}} \psi_+(\bd{q^{k}},
    \bd{p^{k,\ell}}) e^{\frac{\I}{\veps}
      \bd{P}^{\bd{k,\ell}}_+\cdot(\bd{x}^{\bd{n}} -
      \bd{Q}^{\bd{k,\ell}}_+) - \frac{1}{2\veps} \abs{\bd{x}^{\bd{n}}
        - \bd{Q}^{\bd{k,\ell}}_+}^2}  \\
    &\qquad + \frac{a_-^{\bd{k,\ell}}r_{\theta}^-}{(2\pi \veps)^{3d/2}}
    \psi_-(\bd{q^{k}}, \bd{p^{k,\ell}}) e^{\frac{\I}{\veps}
      \bd{P}^{\bd{k,\ell}}_-\cdot(\bd{x}^{\bd{n}} -
      \bd{Q}^{\bd{k,\ell}}_-) - \frac{1}{2\veps}
      \abs{\bd{x}^{\bd{n}} - \bd{Q}^{\bd{k,\ell}}_-}^2}  \biggr) \\
    & \qquad \times  \delta q_1\cdots \delta q_d \delta
    p_1\cdots \delta p_d,
  \end{aligned}
\end{equation}
where $r_{\theta}^\pm=r_{\theta}(\abs{\bd{x}^{\bd{n}} -
      \bd{Q}^{\bd{k,\ell}}_{\pm}})$.
\end{itemize}

\begin{remark}

  1. In setting up the meshes, we assume that the initial
  condition \eqref{eq:WKBini} either has compact support or decays
  sufficiently fast to zero as $\bd{x}\rightarrow \infty$ so that we
  only need finite number of mesh points in physical space.

  2. The role of the truncation function $r_{\theta}$ is to save
  computational cost, since although a Gaussian function is not
  localized, it decays quickly away from the center. In practice we
  take $\theta=\Or(\sqrt{\veps})$, the same order as the width of each
  Gaussian, when we evaluate \eqref{eq:discpsi} and \eqref{eq:discFGB}
  numerically.

  3. There are two types of errors present in the method. The first
  type comes from the asymptotic approximation to the wave
  equation. This error can {\it not} be reduced unless one includes
  higher order corrections. The other type is the numerical error
  which comes from two sources: one is from the ODE numerical
  integrator; the other is from the discrete approximations of
  integrals \eqref{eq:recons} and \eqref{eq:psi}. It {\it can} be
  reduced by either taking small mesh size and time step or using
  higher order numerical methods.

  4. Note that the assumption that the initial conditions are either
  compactly supported or decay quickly implies that the values on the
  boundary are zero (or close to zero). Then \eqref{eq:discpsi} and
  \eqref{eq:discFGB} are the trapezoidal rules to approximate
  \eqref{eq:psi} and \eqref{eq:recons}. Notice that, due to the
  Gaussian factor, the integrand functions in \eqref{eq:psi} and
  \eqref{eq:recons} are exponentially small unless $\bd{x-Q}$ and
  $\bd{y-q}$ are in the order of $\Or({\veps}^{1/2})$, which implies
  their derivatives with respect to $\bd{y},\;\bd{q},\;\bd{p}$ are of
  the order $\Or(\veps^{-1/2})$. This suggests $\bd{\delta y},\;
  \bd{\delta q},\;\bd{\delta p}$ should be taken as the size of
  $\Or(\sqrt{\veps})$.  Hence $N_y$ and $N_q$ are of order
  $\Or(\veps^{-d/2})$. As illustrated in Figure \ref{fig:meshpq},
  $N_p$ is usually taken as $\Or\left(\frac{\sqrt{\veps}}{\min_j
      \;\delta p_j}\right)$, which is of order $\Or(1)$.
  % The systematic discussion on how to take meshes of
  % $\bd{y},\;\bd{q},\;\bd{p}$ is given in Section 3 of
  % \cite{LuYang:MMS}.
  $N_x$ is not constrained by $\veps$, and is only determined by how
  well represented one wants the final solution.

  5. Step $2$ and $4$ can be expedited by making use of the discrete
  fast Gaussian transform, as in \cites{QiYi:app1, QiYi:app2}.
\end{remark}

\section{Asymptotic derivation}\label{sec:asymder}

We now derive the formulation shown in Section \ref{sec:formu} using
asymptotic analysis.

We start with the following ansatz for the wave equation
\eqref{eq:wave},
\begin{equation}\label{eq:ansatz1}
  \begin{aligned}
    u(t, \bd{x}) & = \frac{1}{(2\pi \veps)^{3d/2}} \int_{\RR^{3d}}
    a_+(t, \bd{q}, \bd{p}) e^{\frac{\I}{\veps} \Phi_+(t, \bd{x},
      \bd{y}, \bd{q}, \bd{p})} u_{+,0}(\bd{y}) \ud \bd{y} \ud \bd{p} \ud \bd{q} \\
    & + \frac{1}{(2\pi \veps)^{3d/2}} \int_{\RR^{3d}} a_-(t, \bd{q},
    \bd{p}) e^{\frac{\I}{\veps} \Phi_-(t, \bd{x}, \bd{y}, \bd{q} ,
      \bd{p})} u_{-,0}(\bd{y}) \ud \bd{y} \ud \bd{p} \ud \bd{q},
  \end{aligned}
\end{equation}
where $\Phi_{\pm}$ are given by
\begin{multline}\label{eq:phi1}
  \Phi_{\pm}(t, \bd{x}, \bd{y}, \bd{q}, \bd{p}) = S_{\pm}(t,
  \bd{q}, \bd{p})+\bd{P}_{\pm}(t,
  \bd{q}, \bd{p}) \cdot ( \bd{x} -
  \bd{Q}_{\pm}(t, \bd{q}, \bd{p})) - \bd{p} \cdot ( \bd{y} - \bd{q} ) \\
  + \frac{\I}{2} \abs{\bd{x} - \bd{Q}_{\pm}(t, \bd{q}, \bd{p})}^2 +
  \frac{\I}{2} \abs{\bd{y} - \bd{q}}^2.
\end{multline}

The initial conditions are taken as
\begin{equation}\label{eq:ini1}
\begin{aligned}
 & \bd{Q}_{\pm}(0,\bd{q},\bd{p}) = \bd{q},\qquad
    \bd{P}_{\pm}(0,\bd{q},\bd{p}) =
    \bd{p},\\
    & S_{\pm}(0,\bd{q},\bd{p})=0,\qquad
    a_{\pm}(0,\bd{q},\bd{p})=2^{d/2}.
\end{aligned}
\end{equation}

The subscript $\pm$ indicates the two branches that correspond to
two different Hamiltonian,
\begin{equation}
  H_+(\bd{Q_+}, \bd{P_+}) = c(\bd{Q_+}) \abs{\bd{P_+}}, \qquad
  H_-(\bd{Q_-}, \bd{P_-}) = -c(\bd{Q_-}) \abs{\bd{P_-}}.
\end{equation}
$\bd{P}_{\pm}$ and $\bd{Q}_{\pm}$ satisfy the equation of motion
given by the Hamiltonian $H_{\pm}$
\begin{equation}\label{eq:characline1}
  \begin{cases}
    \displaystyle
    \partial_t \bd{Q}_{\pm} = \partial_{\bd{P}_{\pm}} H_{\pm} = \pm c
    \frac{\bd{P}_{\pm}}{\abs{\bd{P}_{\pm}}},\\
    \displaystyle
    \partial_t \bd{P}_{\pm} = - \partial_{\bd{Q}_{\pm}} H_{\pm}
    = \mp \partial_{\bd{Q}_{\pm}} c \abs{\bd{P}_{\pm}}.
  \end{cases}
\end{equation}

By plugging \eqref{eq:ansatz1} into \eqref{eq:wave}, the leading
order terms show that the evolution of $S_{\pm}$ simply satisfies
\begin{equation}\label{eq:Spm}
  {\partial_t S_{\pm}}=0.
\end{equation}
This implies $S_{\pm}(t, \bd{q}, \bd{p})=0$. Hence we omit the terms
$S_{\pm}$ in Section \ref{sec:formu} and later calculations.

Before proceeding further, let us state some lemmas that will be
used.
\begin{lemma}\label{lem:FBI}
  For $u \in L^2(\RR^d)$, it holds
  \begin{equation}\label{eq:FBI}
    u(\bd{x}) = \frac{1}{(2\pi \veps)^{3d/2}}
    \int_{\RR^{3d}} 2^{d/2} e^{\frac{\I}{\veps} \Phi_{\pm}(0, \bd{x}, \bd{y}, \bd{q}, \bd{p})}
    u(\bd{y}) \ud \bd{y}\ud \bd{p} \ud \bd{q}.
  \end{equation}
\end{lemma}

\begin{proof}
  By the initial conditions \eqref{eq:ini1},
  \begin{equation}
    \Phi_{\pm}(0, \bd{x}, \bd{y}, \bd{q}, \bd{p})
    = \bd{p} \cdot ( \bd{x} - \bd{q} ) - \bd{p} \cdot ( \bd{y} - \bd{q} )
    + \frac{\I}{2} \abs{ \bd{x} - \bd{q} }^2
    + \frac{\I}{2} \abs{ \bd{y} - \bd{q} }^2.
  \end{equation}
  Therefore, \eqref{eq:FBI} is just the standard wave packet
  decomposition in disguise (see for example \cite{Fo:89}).
\end{proof}

The proof of the following important lemma follows the one of Lemma
$3$ in \cite{SwRo:09}.
\begin{lemma}\label{lem:veps1}
  For any vector $\bd{a}(\bd{y},\bd{q},\bd{p})$ and matrix
  $M(\bd{y},\bd{q},\bd{p})$ in Schwartz class viewed as functions of
  $(\bd{y},\bd{q},\bd{p})$, we have
  \begin{equation}\label{eq:con1}
    \bd{a}(\bd{y},\bd{q},\bd{p})
    \cdot (\bd{x} - \bd{Q})  \sim - \veps \partial_{z_k} ( a_j
    Z_{jk}^{-1} ),
  \end{equation}
  and
  \begin{equation}\label{eq:con2}
    (\bd{x} - \bd{Q})\cdot M(\bd{y},\bd{q},\bd{p}) (\bd{x} - \bd{Q})
    \sim \veps \partial_{z_l} Q_j M_{jk} Z_{kl}^{-1}   + \veps^2
    \partial_{z_m} \bigl( \partial_{z_l} (M_{jk} Z_{kl}^{-1} ) Z_{jm}^{-1}
    \bigr),
  \end{equation}
  where Einstein's summation convention has been used.

  Moreover, for multi-index $\alpha$ that $\abs{\alpha} \geq 3$,
  \begin{equation}\label{eq:con3}
    (\bd{x} - \bd{Q})^{\alpha} \sim \Or(\veps^{\abs{\alpha}-1}).
  \end{equation}
  Here we use the notation $ f \sim g $ to mean that
  \begin{equation}
    \int_{\RR^{3d}} f e^{\frac{\I}{\veps} \Phi_{\pm}} \ud \bd{y} \ud \bd{p} \ud \bd{q}
    = \int_{\RR^{3d}} g e^{\frac{\I}{\veps} \Phi_{\pm}} \ud \bd{y} \ud \bd{p} \ud
    \bd{q}.
  \end{equation}
\end{lemma}

\begin{proof}
  Since the proof is exactly the same for the cases of $\Phi_+$ and
  $\Phi_-$, we omit the subscript $\pm$ for simplicity. As $\bd{a}$
  and $M$ are in Schwartz class, all the manipulations below are
  justified.

  Observe that at $t=0$,
  \begin{equation*}
    -(\partial_{\bd{q}}\bd{Q}) \bd{P}+\bd{p}=0,
    \qquad (\partial_{\bd{p}}\bd{Q}) \bd{P}=0.
  \end{equation*}
  Using \eqref{eq:characline1}, we have
  \begin{align*}
    \partial_t\bigl(-(\partial_{\bd{q}}\bd{Q})
    \bd{P}+\bd{p}\bigr)&=-\partial_{\bd{q}}\left(\partial_t \bd{Q}
    \right)\bd{P}-\partial_{\bd{q}}\bd{Q} \partial_t \bd{P}\\
    &=-\partial_{\bd{q}}\left(c\frac{\bd{P}}{\abs{\bd{P}}}
    \right)\bd{P}+\partial_{\bd{q}}\bd{Q}\partial_{\bd{Q}}c\abs{\bd{P}}\\
    &=0.
  \end{align*}
  Analogously we have $\displaystyle
  \partial_t\bigl((\partial_{\bd{p}}\bd{Q}) \bd{P}\bigr)=0$. Therefore
  for all $t>0$,
  \[-(\partial_{\bd{q}}\bd{Q}) \bd{P}+\bd{p}=0,\qquad
  (\partial_{\bd{p}}\bd{Q}) \bd{P}=0.\]

  Then straightforward calculations yield
  \begin{align*}
    & \partial_{\bd{q}}\Phi=(\partial_{\bd{q}}\bd{P}
    -\I\partial_{\bd{q}}\bd{Q})(\bd{x}-\bd{Q})-\I (\bd{y}-\bd{q}), \\
    & \partial_{\bd{p}}\Phi=(\partial_{\bd{p}}\bd{P}
    -\I\partial_{\bd{p}}\bd{Q})(\bd{x}-\bd{Q})- (\bd{y}-\bd{q}),
  \end{align*}
  which implies that
  \begin{equation}
    \I\partial_{\bd{z}}\Phi=Z(\bd{x}-\bd{Q}),
  \end{equation}
  where $\partial_{\bd{z}}$ and $Z$ are defined in \eqref{eq:op_zZ}.
 Note that, $Z$ can be rewritten as
  \begin{equation*}
    Z = \partial_{\bd{z}}( \bd{Q} + \I \bd{P} )
    =
    \begin{pmatrix}
      \I I & I
    \end{pmatrix}
    \begin{pmatrix}
      \partial_{\bd{q}} \bd{Q} & \partial_{\bd{q}} \bd{P} \\
      \partial_{\bd{p}} \bd{Q} & \partial_{\bd{p}} \bd{P}
    \end{pmatrix}
    \begin{pmatrix}
      - \I I \\
      I
    \end{pmatrix},
  \end{equation*}
  where $I$ stands for the $d \times d$ identity matrix. Therefore,
  define
  \begin{equation*}
    F =
    \begin{pmatrix}
      \partial_{\bd{q}} \bd{Q} & \partial_{\bd{q}} \bd{P} \\
      \partial_{\bd{p}} \bd{Q} & \partial_{\bd{p}} \bd{P}
    \end{pmatrix},
  \end{equation*}
  then
  \begin{equation*}
    \begin{aligned}
      Z Z^{\ast} & =
      \begin{pmatrix}
        \I I & I
      \end{pmatrix}
      F
      \begin{pmatrix}
        I  & -\I I \\
        \I I & I
      \end{pmatrix}
      F^{\mathrm{T}}
      \begin{pmatrix}
        - \I I \\
        I
      \end{pmatrix} \\
      & =
      \begin{pmatrix}
        \I I & I
      \end{pmatrix}
      F
      F^{\mathrm{T}}
      \begin{pmatrix}
        - \I I \\
        I
      \end{pmatrix} +
      \begin{pmatrix}
        \I I & I
      \end{pmatrix}
      F
      \begin{pmatrix}
        0 & -\I I \\
        \I I & 0
      \end{pmatrix}
      F^{\mathrm{T}}
      \begin{pmatrix}
        - \I I \\
        I
      \end{pmatrix}  \\
      & =
      \begin{pmatrix}
        \I I & I
      \end{pmatrix}
      F
      F^{\mathrm{T}}
      \begin{pmatrix}
        - \I I \\
        I
      \end{pmatrix} +
      2I.
    \end{aligned}
  \end{equation*}
  In the last equality, we have used the fact that
  \begin{equation*}
    F
    \begin{pmatrix}
      0 & -\I I \\
      \I I & 0
    \end{pmatrix}
    F^{\mathrm{T}} =
    \begin{pmatrix}
      0 & -\I I \\
      \I I & 0
    \end{pmatrix},
  \end{equation*}
  due to the Hamiltonian flow structure. Therefore $ZZ^{\ast}$ is
  positive definite for all $t$, which implies $Z$ is invertible and
  \begin{equation}\label{eq:dzPhi}
    (\bd{x} - \bd{Q}) = \I Z^{-1} \partial_{\bd{z}} \Phi.
  \end{equation}

  Using \eqref{eq:dzPhi}, one has
  \begin{align*}
    \int_{\RR^{3d}} \bd{a} \cdot (\bd{x} - \bd{Q}) e^{\frac{\I}{\veps}
      \Phi} \ud \bd{y} \ud \bd{p} \ud \bd{q} &=\veps\int_{\RR^{3d}}
    {a}_j Z^{-1}_{jk} \left(\frac{\I}{\veps}\partial_{z_k}\Phi\right)
    e^{\frac{\I}{\veps} \Phi} \ud \bd{y} \ud \bd{p} \ud \bd{q} \\
    &=-\veps \int_{\RR^{3d}} \partial_{z_k}\big({a}_j Z^{-1}_{jk}
    \big) e^{\frac{\I}{\veps} \Phi} \ud \bd{y} \ud \bd{p} \ud \bd{q},
  \end{align*}
  where the last equality is obtained from integration by parts. This
  proves \eqref{eq:con1}.

  Making use of \eqref{eq:con1} twice produces \eqref{eq:con2}
  \begin{align*}
    (\bd{x} - \bd{Q})\cdot M (\bd{x} - \bd{Q})
    & = (x-Q)_jM_{jk}(x-Q)_k \\
    & \sim -\veps \partial_{z_l}\bigl((x-Q)_jM_{jk}Z^{-1}_{kl}
    \bigr) \\
    & = \veps \partial_{z_l} Q_j M_{jk} Z_{kl}^{-1}
    -\veps(x-Q)_j\partial_{z_l}(M_{jk}Z^{-1}_{kl})
    \\
    & \sim \veps \partial_{z_l} Q_j M_{jk} Z_{kl}^{-1} + \veps^2
    \partial_{z_m} \bigl( \partial_{z_l} (M_{jk} Z_{kl}^{-1} ) Z_{jm}^{-1} \bigr).
  \end{align*}

  By induction it is easy to see that \eqref{eq:con3} is true.

\end{proof}

\subsection{Initial value decomposition}

By \eqref{eq:phi1} and \eqref{eq:characline1} we obtain that
\begin{equation}\label{eq:partialtPhi}
  \begin{aligned}
    \partial_t \Phi_{\pm} & = - \bd{P}_{\pm} \cdot \partial_t
    \bd{Q}_{\pm} + (\partial_t \bd{P}_{\pm} -
    \I \partial_t \bd{Q}_{\pm}) \cdot (\bd{x} - \bd{Q}_{\pm}) \\
    & = \mp c \abs{\bd{P}_{\pm}} \mp (\bd{x} - \bd{Q}_{\pm}) \cdot
    \Bigl( \abs{\bd{P}_{\pm}} \partial_{\bd{Q}_{\pm}} c + \I
    \frac{\bd{P}_{\pm}}{\abs{\bd{P}_{\pm}}} c \Bigr),
  \end{aligned}
\end{equation}
and in particular for $t = 0$,
\begin{equation}
  \partial_t \Phi_{\pm}(0, \bd{x}, \bd{y}, \bd{q}, \bd{p})
  = \mp c \abs{\bd{p}} \mp (\bd{x} - \bd{q}) \cdot \Bigl(
    \abs{\bd{p}} \partial_{\bd{q}} c + \I
    \frac{\bd{p}}{\abs{\bd{p}}} c \Bigr).
\end{equation}

The ansatz \eqref{eq:ansatz1} shows that
\begin{equation}\label{eq:initialu}
  \begin{aligned}
    u(0, \bd{x}) & = \frac{1}{(2\pi \veps)^{3d/2}} \int_{\RR^{3d}}
    a_+(0, \bd{q}, \bd{p}) e^{\frac{\I}{\veps} \Phi_+(0, \bd{x},
      \bd{y}, \bd{q}, \bd{p})} u_{+,0}(\bd{y})
    \ud \bd{y} \ud \bd{p} \ud \bd{q} \\
    & + \frac{1}{(2\pi \veps)^{3d/2}} \int_{\RR^{3d}} a_-(0, \bd{p},
    \bd{q}) e^{\frac{\I}{\veps} \Phi_-(0, \bd{x}, \bd{y}, \bd{q} ,
      \bd{p})} u_{-,0}(\bd{y}) \ud \bd{y} \ud \bd{p} \ud \bd{q},
  \end{aligned}
\end{equation}
and
\begin{equation}\label{eq:initialv}
  \begin{aligned}
    \partial_t u(0, \bd{x}) & = \frac{1}{(2\pi \veps)^{3d/2}}
    \int_{\RR^{3d}} \bigl(\partial_t a_+ +
    \frac{\I a_+}{\veps} \partial_t \Phi_+\bigr) e^{\frac{\I}{\veps}
      \Phi_+(0, \bd{x}, \bd{y}, \bd{q}, \bd{p})} u_{+,0}(\bd{y}) \ud
    \bd{y} \ud \bd{p} \ud \bd{q} \\
    & + \frac{1}{(2\pi \veps)^{3d/2}} \int_{\RR^{3d}} \bigl(\partial_t
    a_- + \frac{\I a_-}{\veps} \partial_t \Phi_-\bigr)
    e^{\frac{\I}{\veps} \Phi_-(0, \bd{x}, \bd{y}, \bd{q} , \bd{p})}
    u_{-,0}(\bd{y}) \ud \bd{y} \ud \bd{p} \ud \bd{q}.
  \end{aligned}
\end{equation}

We take
\begin{align}
  & a_{\pm}(0, \bd{q}, \bd{p}) = 2^{d/2}, \label{eq:a0}\\
  & u_{+,0}(\bd{x}) = A_+(\bd{x}) e^{\frac{\I}{\veps} S_0(\bd{x})}, \\
  & u_{-,0}(\bd{x}) = A_-(\bd{x}) e^{\frac{\I}{\veps} S_0(\bd{x})},
\end{align}
with
\begin{equation}\label{eq:Apm}
  A_{\pm}(\bd{x}) = \frac{1}{2} \biggl( A_0(\bd{x}) \pm
  \frac{\I B_0(\bd{x})}{ c(\bd{x}) \abs{ \partial_{\bd{x}} S_0(\bd{x})}}
  \biggr).
\end{equation}
We next show that this will approximate the initial condition to the
leading order in $\veps$.

Substituting \eqref{eq:a0}-\eqref{eq:Apm} into \eqref{eq:initialu}
and using Lemma~\ref{lem:FBI}, we easily confirm that
\begin{equation*}
  u(0, \bd{x}) = u_{+,0}(\bd{x}) + u_{-,0}(\bd{x})
  = A_0(\bd{x}) e^{\frac{\I}{\veps} S_0(\bd{x})}.
\end{equation*}
For the initial velocity, we substitute \eqref{eq:a0}-\eqref{eq:Apm}
into \eqref{eq:initialv} and keep only the leading order terms in
$\veps$. According to Lemma~\ref{lem:veps1}, only the term $\mp
c\abs{\bd{p}}$ in $\partial_t \Phi_{\pm}$ will contribute to the
leading order, since the other terms that contain $(\bd{x} -
\bd{q})$ are $\Or(\veps)$. Hence,
\begin{equation*}
  \begin{aligned}
    \partial_t u(0, \bd{x}) = & - \frac{2^{d/2}}{(2\pi \veps)^{3d/2}}
    \int_{\RR^{3d}}  \frac{\I}{\veps} c(\bd{q})\abs{\bd{p}}
    e^{\frac{\I}{\veps} \Phi_+(0, \bd{x}, \bd{y}, \bd{q}, \bd{p})}
    u_{+,0}(\bd{y}) \ud
    \bd{y} \ud \bd{p} \ud \bd{q} \\
    & + \frac{2^{d/2}}{(2\pi \veps)^{3d/2}} \int_{\RR^{3d}} \frac{\I}{\veps}
    c(\bd{q}) \abs{\bd{p}} e^{\frac{\I}{\veps} \Phi_-(0, \bd{x},
      \bd{y}, \bd{q} , \bd{p})} u_{-,0}(\bd{y}) \ud \bd{y} \ud \bd{p}
    \ud \bd{q} + \Or(1).
  \end{aligned}
\end{equation*}
Consider the integral
\begin{equation*}
  \int_{\RR^d} c(\bd{q}) \abs{\bd{p}} e^{- \frac{\I}{\veps} \bd{p}
    \cdot (\bd{y} - \bd{q}) - \frac{1}{2\veps} \abs{\bd{y} -
      \bd{q}}^2} A_{\pm}(\bd{y})
  e^{\frac{\I}{\veps} S_0(\bd{y})} \ud \bd{y}
  = \int_{\RR^d} c(\bd{q}) \abs{\bd{p}} A_{\pm}(\bd{y})
  e^{\frac{\I}{\veps} \Theta(\bd{y}, \bd{q}, \bd{p})} \ud \bd{y}.
\end{equation*}
The phase function $\Theta$ is given by
\begin{equation*}
  \Theta(\bd{y}, \bd{q}, \bd{p}) =
  - \bd{p} \cdot (\bd{y} - \bd{q}) + \frac{\I}{2} \abs{\bd{y} - \bd{q}}^2
  + S_0(\bd{y}).
\end{equation*}
Clearly, $\Im \;\Theta \geq 0$ and $\Im \;\Theta = 0$ if and only if $
\bd{y} = \bd{q}$. The derivatives of $\Theta$ with respect to $\bd{y}$
are
\begin{align*}
  & \partial_{\bd{y}} \Theta = - \bd{p} + \partial_{\bd{y}} S_0(\bd{y})
  + \I (\bd{y} - \bd{q}), \\
  & \partial_{\bd{y}}^2 \Theta = \partial_{\bd{y}}^2 S_0(\bd{y}) + i I.
\end{align*}
Hence, the first derivative vanishes only when $\bd{y} = \bd{q}$ and
$\bd{p} = \partial_{\bd{y}} S_0(\bd{y})$, and $\det
\partial_{\bd{y}}^2 \Theta \not = 0$. Therefore, we can apply
stationary phase approximation with complex phase (see for example
\cite{Ho:83}) to conclude, for $(\bd{q}, \bd{p}) \in \RR^{2d}$,
\begin{multline*}
  \int_{\RR^d} c(\bd{q}) \abs{\bd{p}} e^{- \frac{\I}{\veps} \bd{p}
    \cdot (\bd{y} - \bd{q}) - \frac{1}{2\veps} \abs{\bd{y} -
      \bd{q}}^2} A_{\pm}(\bd{y})
  e^{\frac{\I}{\veps} S_0(\bd{y})} \ud \bd{y} \\
  = \int_{\RR^d} c(\bd{y}) \abs{\partial_{\bd{y}} S_0(\bd{y})} e^{-
    \frac{\I}{\veps} \bd{p} \cdot (\bd{y} - \bd{q}) - \frac{1}{2\veps}
    \abs{\bd{y} - \bd{q}}^2} A_{\pm}(\bd{y}) e^{\frac{\I}{\veps}
    S_0(\bd{y})} \ud \bd{y} + \Or(\veps).
\end{multline*}
Therefore,
\begin{multline}\label{eq:dtu0}
  \partial_t u(0, \bd{x}) = - \frac{2^{d/2}}{(2\pi \veps)^{3d/2}}
  \int_{\RR^{3d}} \frac{\I}{\veps} c(\bd{y}) \abs{\partial_{\bd{y}}
    S_0(\bd{y})} e^{\frac{\I}{\veps} \Phi_+(0, \bd{x}, \bd{y}, \bd{q},
    \bd{p})} u_{+,0}(\bd{y}) \ud
  \bd{y} \ud \bd{p} \ud \bd{q} \\
  + \frac{2^{d/2}}{(2\pi \veps)^{3d/2}} \int_{\RR^{3d}} \frac{\I}{\veps}
  c(\bd{y}) \abs{\partial_{\bd{y}} S_0(\bd{y})} e^{\frac{\I}{\veps}
    \Phi_-(0, \bd{x}, \bd{y}, \bd{q} , \bd{p})} u_{-,0}(\bd{y}) \ud
  \bd{y} \ud \bd{p} \ud \bd{q} + \Or(1).
\end{multline}
Substitute \eqref{eq:Apm} into \eqref{eq:dtu0} and use
Lemma~\ref{lem:FBI}, then
\begin{equation*}
  \partial_t u(0, \bd{x}) = \frac{1}{\veps} B_0(\bd{x})
  e^{\frac{\I}{\veps} S_0(\bd{x})},
\end{equation*}
which agrees with \eqref{eq:WKBini}.

\subsection{Derivation of the evolution equation of $a_{\pm}$}

In order to derive the evolution equation for the weight function
$a$, we carry out the asymptotic analysis of the wave equation
\eqref{eq:wave} using the ansatz \eqref{eq:ansatz1} in this section.
As the equation \eqref{eq:wave} is linear, we can deal with the two
branches separately. In the following, we only deal with the ``$+$''
branch that corresponds to $H_+$, and the other is completely
analogous. For simplicity, we drop the subscript ``$+$'' in the
notations.

Substituting \eqref{eq:ansatz1} into the equation \eqref{eq:wave}
(keeping only the ``$+$'' branch) gives
\begin{equation*}
  \partial_t^2 u = \frac{1}{(2\pi \veps)^{3d/2}} \int_{\RR^{3d}}
  \Bigl( \partial_t^2 a + 2 \frac{\I}{\veps} \partial_t a \partial_t \Phi
  + \frac{\I}{\veps} a \partial_t^2 \Phi - \frac{1}{\veps^2} a(\partial_t\Phi)^2
  \Bigr) e^{\I\Phi/\veps} u_0 \ud \bd{y} \ud \bd{p} \ud \bd{q},
\end{equation*}
and
\begin{equation*}
  \Delta u = \frac{1}{(2\pi \veps)^{3d/2}} \int_{\RR^{3d}}
  \Bigl( \frac{\I}{\veps}  \Delta \Phi - \frac{1}{\veps^2}
  (\partial_{\bd{x}}\Phi \cdot \partial_{\bd{x}} \Phi )
  \Bigr) a e^{\I\Phi/\veps} u_0 \ud \bd{y} \ud \bd{p} \ud \bd{q}.
\end{equation*}

Squaring both sides of \eqref{eq:partialtPhi} yields
\begin{multline}
  (\partial_t \Phi)^2 = c^2 \abs{\bd{P}}^2 + \Bigl( (\bd{x} - \bd{Q})
  \cdot \Bigl( \abs{\bd{P}} \partial_{\bd{Q}} c + \I c
  \frac{\bd{P}}{\abs{\bd{P}}} \Bigr)\Bigr)^2 \\
  + 2 c \abs{\bd{P}} (\bd{x} - \bd{Q}) \cdot \Bigl(
  \abs{\bd{P}} \partial_{\bd{Q}} c + \I c \frac{\bd{P}}{\abs{\bd{P}}}
  \Bigr).
\end{multline}
Differentiating \eqref{eq:partialtPhi} with respect to $t$, one has
\begin{equation}
  \begin{aligned}
    \partial_t^2 \Phi & = - \partial_t(\abs{\bd{P}} c) + \partial_t
    \bd{Q} \cdot \Bigl( \abs{\bd{P}}\partial_{\bd{Q}} c
    + \I \frac{\bd{P}}{\abs{\bd{P}}} c \Bigr) \\
    & - ( \bd{x} - \bd{Q} ) \cdot \Bigl( \partial_{\bd{Q}} c
    \frac{\bd{P}\cdot \partial_t \bd{P}}{\abs{\bd{P}}}
    + \partial_{\bd{Q}}^2 c \cdot \partial_t \bd{Q} \abs{\bd{P}} \\
    & \hspace{6em} + \I c \frac{\partial_t \bd{P}}{\abs{\bd{P}}}
    - \I c \bd{P} \frac{\bd{P} \cdot \partial_t \bd{P}}{\abs{\bd{P}}^3}
    + \I \frac{\bd{P}}{\abs{\bd{P}}} \partial_{\bd{Q}} c \cdot \partial_t \bd{Q}
    \Bigr).
  \end{aligned}
\end{equation}
We simplify the last equation using \eqref{eq:characline},
\begin{equation}
  \begin{aligned}
    \partial_t^2 \Phi & = c \bd{P} \cdot \partial_{\bd{Q}} c + \I c^2 \\
    & - (\bd{x} - \bd{Q}) \cdot \Bigl(- \partial_{\bd{Q}} c \bd{P}
    \cdot \partial_{\bd{Q}} c + c \partial_{\bd{Q}}^2 c \cdot \bd{P} \\
    & \hspace{6em} - \I c \partial_{\bd{Q}} c + 2\I c \bd{P}
    \frac{\bd{P} \cdot \partial_{\bd{Q}} c} {\abs{\bd{P}}^2}\Bigr).
  \end{aligned}
\end{equation}

Taking derivatives with respect to $\bd{x}$ produces
\begin{equation}
  \partial_{\bd{x}} \Phi = \bd{P} + \I (\bd{x} - \bd{Q}),
\end{equation}
\begin{equation}
  \partial_{\bd{x}} \Phi \cdot \partial_{\bd{x}} \Phi
  = \abs{\bd{P}}^2 + 2 \I \bd{P} \cdot (\bd{x} - \bd{Q})
  - \abs{ \bd{x} - \bd{Q} }^2,
\end{equation}
and
\begin{equation}
  \Delta \Phi = d \I.
\end{equation}

We next expand $c(\bd{x})$ around the point $\bd{Q}$,
\begin{equation}
  c(\bd{x}) = c + \partial_{\bd{Q}} c \cdot ( \bd{x} - \bd{Q} )
  + \frac{1}{2} (\bd{x} - \bd{Q}) \cdot \partial_{\bd{Q}}^2 c
  (\bd{x} - \bd{Q}) + \Or(\abs{\bd{x} - \bd{Q}})^3,
\end{equation}
and
\begin{multline}
  c^2(\bd{x}) = c^2 + 2 c \partial_{\bd{Q}} c \cdot ( \bd{x} - \bd{Q}
  ) + (\partial_{\bd{Q}} c \cdot (\bd{x} - \bd{Q}))^2 \\
  + c (\bd{x} - \bd{Q}) \cdot \partial_{\bd{Q}}^2 c  (\bd{x} -
  \bd{Q}) + \Or(\abs{\bd{x} - \bd{Q}})^3.
\end{multline}
The terms $c$, $\partial_{\bd{Q}} c$ and $\partial_{\bd{Q}}^2 c$ on
the right hand sides are all evaluated at $\bd{Q}$.

Substituting all the above into the wave equation \eqref{eq:wave}
and keeping only the leading order terms give
\begin{equation*}
  \begin{aligned}
    & 2 \frac{\I}{\veps} \partial_t a (- c \abs{\bd{P}} ) u +
    \frac{\I}{\veps} a ( c \bd{P}\cdot \partial_{\bd{Q}} c + \I c^2 ) u \\
    & - \frac{1}{\veps^2} a \Bigl( 2 c (\bd{x} - \bd{Q}) \cdot
    (\abs{\bd{P}}^2 \partial_{\bd{Q}} c + \I c \bd{P}) + \bigl(
    (\bd{x} - \bd{Q}) \cdot (\abs{\bd{P}} \partial_{\bd{Q}}c
    + \I c \bd{P} / \abs{\bd{P}}) \bigr)^2 \Bigr) u \\
    & - c^2 \Bigl( - \frac{1}{\veps} ad - \frac{2\I}{\veps^2} a
    \bd{P}\cdot (\bd{x} - \bd{Q}) + \frac{1}{\veps^2} a \abs{\bd{x} -
      \bd{Q}}^2\Bigr) u \\
    & + \frac{2}{\veps^2} a c \partial_{\bd{Q}} c \cdot (\bd{x} -
    \bd{Q})
    (\abs{\bd{P}}^2 + 2\I \bd{P} \cdot ( \bd{x} - \bd{Q} )) u  \\
    & + \frac{1}{\veps^2} a \abs{\bd{P}}^2 \bigl( (\partial_{\bd{Q}}c
    \cdot (\bd{x} - \bd{Q}))^2 u + c (\bd{x} - \bd{Q}) \cdot
    \partial_{\bd{Q}}^2 c (\bd{x} - \bd{Q}) \bigr) u \sim \Or(1).
  \end{aligned}
\end{equation*}
After reorganizing the terms, we get
\begin{equation}
  2 \frac{\I}{\veps} c \abs{\bd{P}} \partial_t a u
  \sim \frac{\I}{\veps} a ( c\bd{P}\cdot \partial_{\bd{Q}} c - (d - 1)c^2 \I) u
  - \frac{1}{\veps^2} a (\bd{x} - \bd{Q}) \cdot M (\bd{x} - \bd{Q}) u,
\end{equation}
where
\begin{multline}
  M = (\abs{\bd{P}} \partial_{\bd{Q}}c - \I c \bd{P}/ \abs{\bd{P}})
  \otimes (\abs{\bd{P}} \partial_{\bd{Q}}c - \I c \bd{P}/
  \abs{\bd{P}}) + c^2 I \\
  - \abs{\bd{P}}^2 \partial_{\bd{Q}} c \otimes
  \partial_{\bd{Q}} c - \abs{\bd{P}}^2 c \partial_{\bd{Q}}^2 c.
\end{multline}
Lemma \ref{lem:veps1} shows that
\begin{equation}\label{eq:quadform}
  a (\bd{x} - \bd{Q}) \cdot M (\bd{x} - \bd{Q}) u \sim
  \veps a \tr(Z^{-1} \partial_{\bd{z}} \bd{Q} M) u + \Or(\veps^2).
\end{equation}
Therefore, to the leading order, we obtain the evolution equation of
$a$,
\begin{multline}\label{eq:aeqn}
  \partial_t a = \frac{a}{2} \Bigl( \frac{\bd{P}}{\abs{\bd{P}}}\cdot
  \partial_{\bd{Q}} c - \frac{(d-1)\I}{\abs{\bd{P}}} c \Bigr) \\
  + \frac{a}{2} \tr\biggl( Z^{-1} \partial_{\bd{z}} \bd{Q} \Bigl( 2
  \frac{\bd{P}}{\abs{\bd{P}}} \otimes \partial_{\bd{Q}} c - \frac{\I
    c}{\abs{\bd{P}}} \Bigl( \frac{\bd{P}\otimes
    \bd{P}}{\abs{\bd{P}}^2} - I \Bigr) - \I
  \abs{\bd{P}}\partial_{\bd{Q}}^2 c \Bigr) \biggr).
\end{multline}
Notice that
\begin{align*}
&\frac{\ud Z}{\ud t}=\partial_{\bd{z}}\left( \frac{\ud \bd{Q}}{\ud
t}+\I\frac{\ud \bd{P}}{\ud t}
\right)=\partial_{\bd{z}}\left(c\frac{\bd{P}}{\abs{\bd{P}}}-\I
\partial_{\bd{Q}}c\abs{\bd{P}}\right) \\ &\qquad =\partial_{\bd{z}}\bd{Q}\frac{
\partial_{\bd{Q}}c\otimes
\bd{P}}{\abs{\bd{P}}}+c\partial_{\bd{z}}\bd{P}
\Bigl( \frac{I}{\abs{\bd{P}}}-\frac{\bd{P}\otimes
    \bd{P}}{\abs{\bd{P}}^3}
    \Bigr)-\I\partial_{\bd{z}}\bd{Q}\partial^2_{\bd{Q}}c\abs{\bd{P}}-
    \I\partial_{\bd{z}}\bd{P}\frac{\bd{P}\otimes
    \partial_{\bd{Q}}c}{\abs{\bd{P}}},\end{align*}and
\begin{equation*}-\frac{(d-1)\I}{\abs{\bd{P}}}c=\tr\bigg(Z^{-1}(\partial_{\bd{z}}\bd{Q}
+\I\partial_{\bd{z}}\bd{P})\frac{\I
c}{\abs{\bd{P}}}\Big(\frac{\bd{P}\otimes\bd{P}}{\abs{\bd{P}}^2}-I\Big)
\bigg).
\end{equation*}

By using the fact that \eqref{eq:quadform} has a quadratic form, one
has
\[\tr\bigg(Z^{-1}\partial_{\bd{z}}\bd{Q}\frac{\bd{P}}{\abs{\bd{P}}}
\otimes\partial_{\bd{Q}}c\bigg)=\tr\bigg(Z^{-1}\partial_{\bd{z}}\bd{Q}
\frac{\partial_{\bd{Q}}c}{\abs{\bd{P}}} \otimes\bd{P}\bigg).\]

Hence \eqref{eq:aeqn} can be reformulated as
\begin{equation*} \frac{\ud a}{\ud
t}=a\frac{\bd{P}}{\abs{\bd{P}}}\cdot
  \partial_{\bd{Q}} c+\frac{a}{2}\tr\left(Z^{-1}\frac{\ud Z}{\ud t} \right).
\end{equation*}

This completes the asymptotic derivation. We remark that in the case
of time dependent Schr\"odinger equation, the asymptotics have been
made rigorous in \cites{SwRo:09, Ro:09}.

\section{Numerical examples}\label{sec:numer}
In this section, we give both one and two dimensional numerical
examples to justify the accuracy of the frozen Gaussian
approximation (FGA). Without loss of generality, we only consider
the wave propagation determined by the ``$+$'' branch of
\eqref{eq:ansatz} which implies that
$B_0(\bd{x})=-\I c(\bd{x})\abs{\nabla_{\bd{x}}S_0(\bd{x})}A_0(\bd{x})$ in
\eqref{eq:WKBini}.

\subsection{One dimension}
Using one-dimensional examples in this section, we compare FGA with
the Gaussian beam method (GBM) in both the accuracy and the
performance when beams spread in GBM. We denote the solution of GBM as
$u^{\GBM}$, and summarize its discrete numerical formulation (only the
``$+$'' branch) as follows for readers' convenience
(\cites{Ra:82,Ta:08,LiRa:09}),
\begin{multline*}
  u^{\GBM}(t,x)=\sum_{j=1}^{N_{y_0}}\left(\frac{1}{2\pi\epsilon}\right)^\frac{1}{2}
  r_{\theta}(\abs{x-y_j})A(t,y_j)\\
  \times \exp\left(\frac{\I}{\veps}
    (S(t,y_j)+\xi(t,y_j)(x-y_j)+M(t,y_j)(x-y_j)^2/2) \right)\delta
  y_0,
\end{multline*} and $y_j,\;\xi,\;S,\;M,\;A$ satisfy
\begin{align*}
  & \frac{\ud y_j}{\ud t}=c(y_j)\frac{\xi}{\abs{\xi}},\qquad
  y_j(0)=y_0^j, \\
  & \frac{\ud \xi}{\ud t}=-\partial_{y_j} c(y_j)\abs{\xi},\qquad
  \xi(0)=\partial_{y_j}S_0(y_j),\\
  & \frac{\ud S}{\ud t}=0,\qquad
  S(0)=S_0(y_j), \\
  & \frac{\ud M}{\ud
    t}=-2\partial_{y_j}c(y_j)\frac{\xi}{\abs{\xi}}M-\partial^2_{y_j}c(y_j)
  \abs{\xi},\qquad M(0)=\partial_{y_j}^2S_0(y_j)+\I, \\
  & \frac{\ud A}{\ud
    t}=\frac{1}{2}\partial_{y_j}c(y_j)\frac{\xi}{\abs{\xi}}A,\qquad
  A(0)=A_0(y_j),
\end{align*} where $r_{\theta}$ is the cutoff function,
$y_0^j$'s are the equidistant mesh points, $\delta y_0$ is the mesh
size and $N_{y_0}$ is the total number of the beams initially
centered at $y_0^j$.

\begin{example}\label{exa:1}

The wave speed is $c(x)=x^2$. The initial conditions are
\begin{align*}
&u_0=\exp\bigl(-100(x-0.5)^2\bigr)\exp\left(\frac{\I x}{\veps}\right),\\
&\partial_t
u_0=-\frac{\I x^2}{\veps}\exp\bigl(-100(x-0.5)^2\bigr)\exp\left(\frac{\I
x}{\veps}\right).
\end{align*}
\end{example}

The final time is $T=0.5$. We plot the real part of the wave field
obtained by FGA compared with the true solution in Figure
\ref{fig:ex1_real} for $\veps=1/64,\;1/128,\;1/256$. The true solution
is computed by the finite difference method using the mesh size of
$\delta x=1/2^{12}$ and the time step of $\delta t=1/2^{18}$ for
domain $[0,2]$. Table \ref{tab:ex1_err} shows the $\ell^\infty$ and
$\ell^2$ errors of both the FGA solution $u^{\FGA}$ and the GBM
solution $u^{\GBM}$. The convergence orders in $\veps$ of
$\ell^\infty$ and $\ell^2$ norms are $1.08$ and $1.17$ separately for
FGA, and $0.54$ and $0.57$ for GBM. We observe a better accuracy order
of FGA than GBM.

We choose $\delta t=1/2^{11}$ for solving the ODEs and $\delta
x=1/2^{12}$ to construct the final solution in both FGA and GBM. In
FGA, we take $\delta q=\delta p=\delta y=1/2^7,\;N_q=128,\;N_p=45$ for
$\veps=1/128,\;1/256$ and $\delta q=\delta p=\delta
y=1/2^5,\;N_q=32,\;N_p=33$ for $\veps=1/64$. In GBM, we take $\delta
y_0=1/2^7,\;N_{y_0}=128$ for $\veps=1/128,\;1/256$ and $\delta
y_0=1/2^5,\;N_{y_0}=32$ for $\veps=1/64$.

We remark that in this example the mesh sizes of $p$ and $q$
have been taken very small and $N_p$ large enough to make sure that
the error of FGA mostly comes from asymptotic expansion, but
not from initial value decomposition, numerical integration of ODEs
and so on.  Such a choice of fine mesh is not necessary for the
accuracy of FGA, as one can see in Example \ref{exa:3}.

\begin{table}
\caption{Example \ref{exa:1}, the $\ell^\infty$ and
  $\ell^2$ errors for FGA and GBM.}\label{tab:ex1_err}
\begin{tabular}{cccc}
\hline \\[-1em] $\veps$ &
${1}/{2^6}$ & ${1}/{2^7}$ &
${1}/{2^8}$ \\ \hline \\[-1em]
$\norm{u-u^{\FGA}}_{\ell^\infty}$ & $1.12\times 10^{-1}$  &
$6.18\times10^{-2}$ &
$2.51\times10^{-2}$\\
\hline
\\[-1em] $\norm{u-u^{\FGA}}_{\ell^2}$& $6.05\times 10^{-2}$ & $2.96\times10^{-2}$ & $1.19\times10^{-2}$ \\
\hline \\[-1em]
$\norm{u-u^{\GBM}}_{\ell^\infty}$ & $7.15\times 10^{-1}$  &
$5.08\times10^{-1}$ &
$3.36\times10^{-1}$\\
\hline
\\[-1em] $\norm{u-u^{\GBM}}_{\ell^2}$& $3.26\times 10^{-1}$ & $2.28\times10^{-1}$ & $1.47\times10^{-1}$ \\ \hline
\end{tabular}
\end{table}

\begin{figure}[h t p]
\begin{tabular}{cc}
\resizebox{2.3in}{!}{\includegraphics{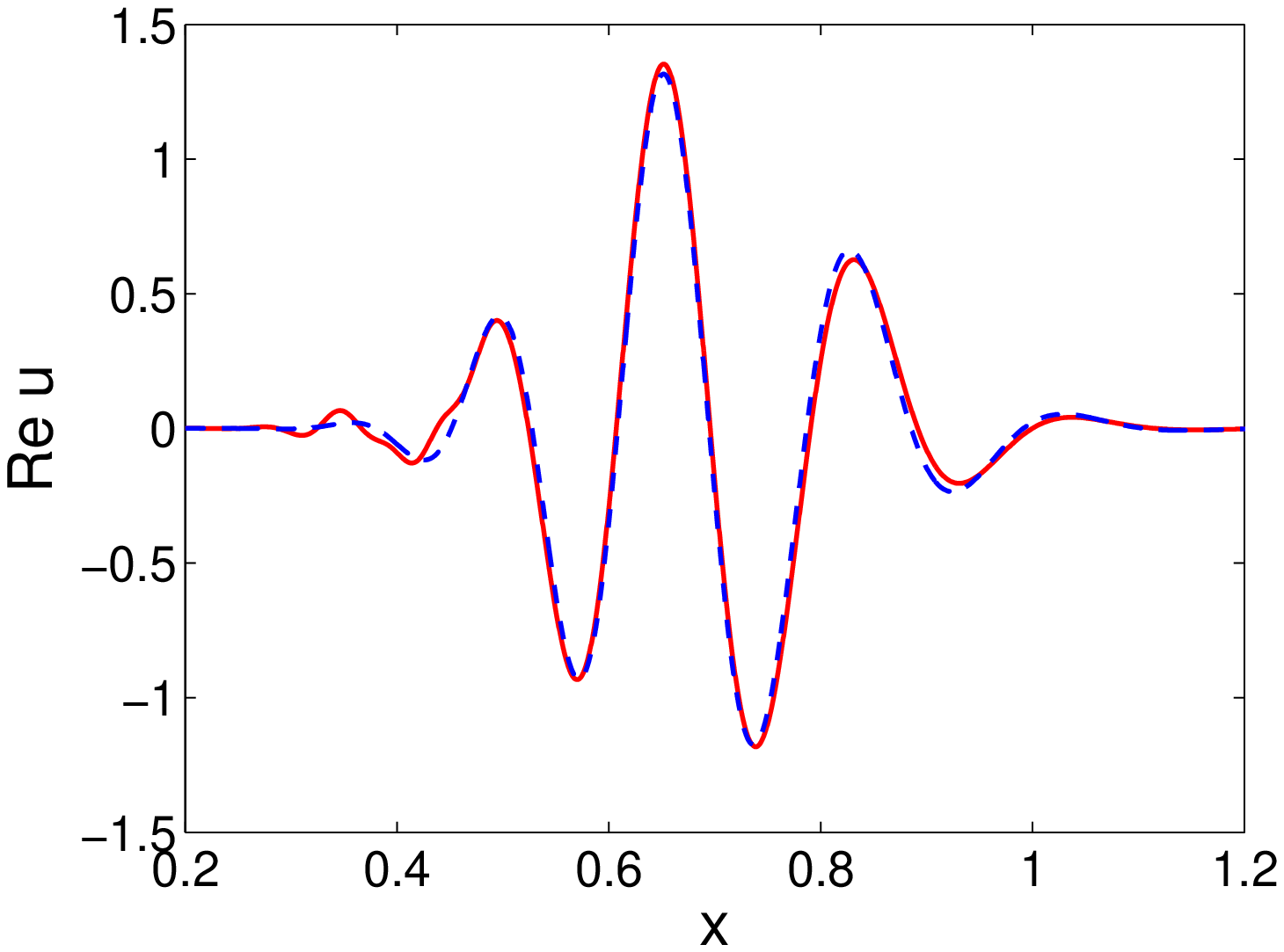}} &
\resizebox{2.3in}{!}{\includegraphics{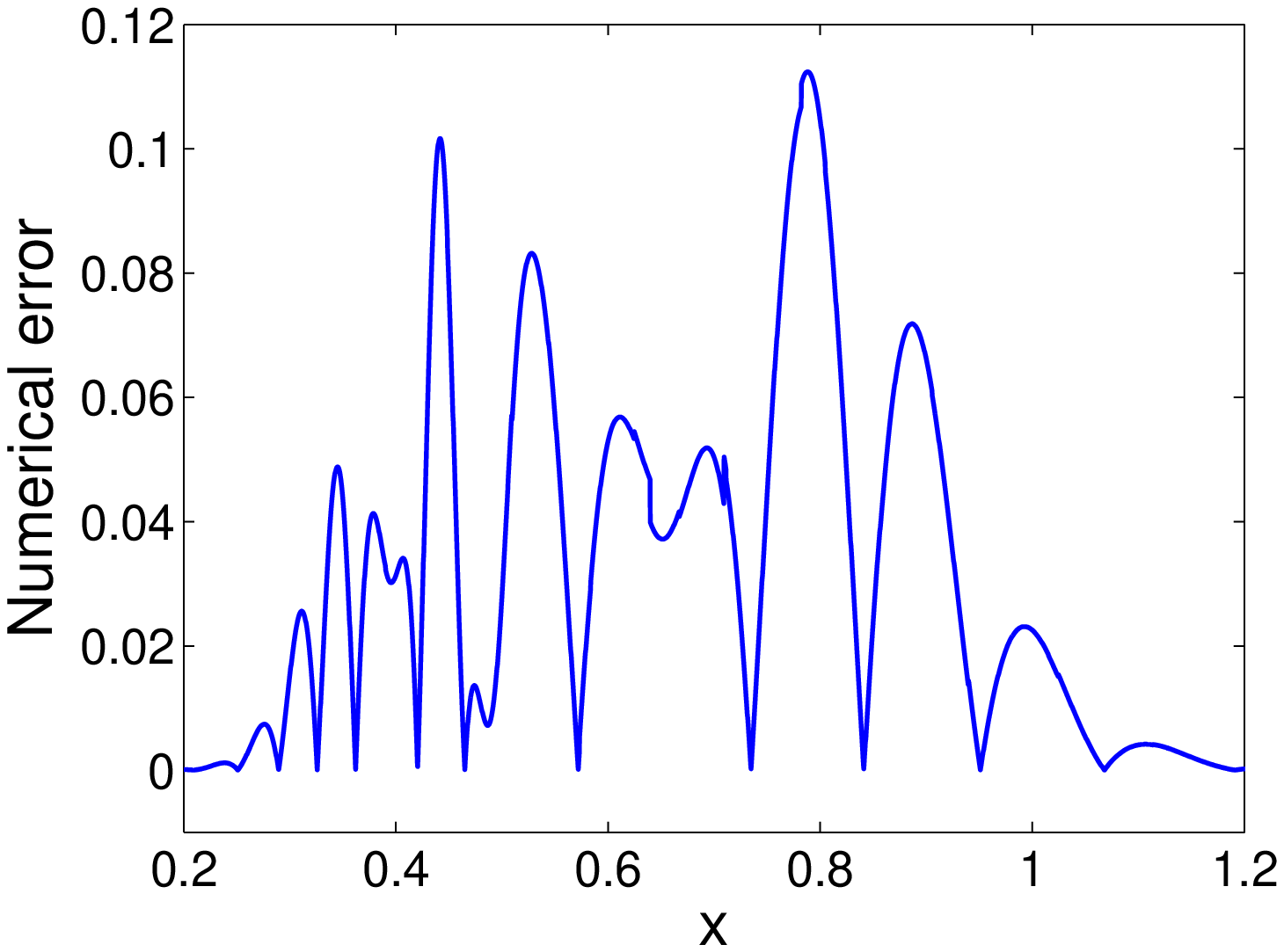}} \\
  \multicolumn{2}{c}{(a) $\veps=\frac{1}{64}$} \\[3mm]
\resizebox{2.3in}{!}{\includegraphics{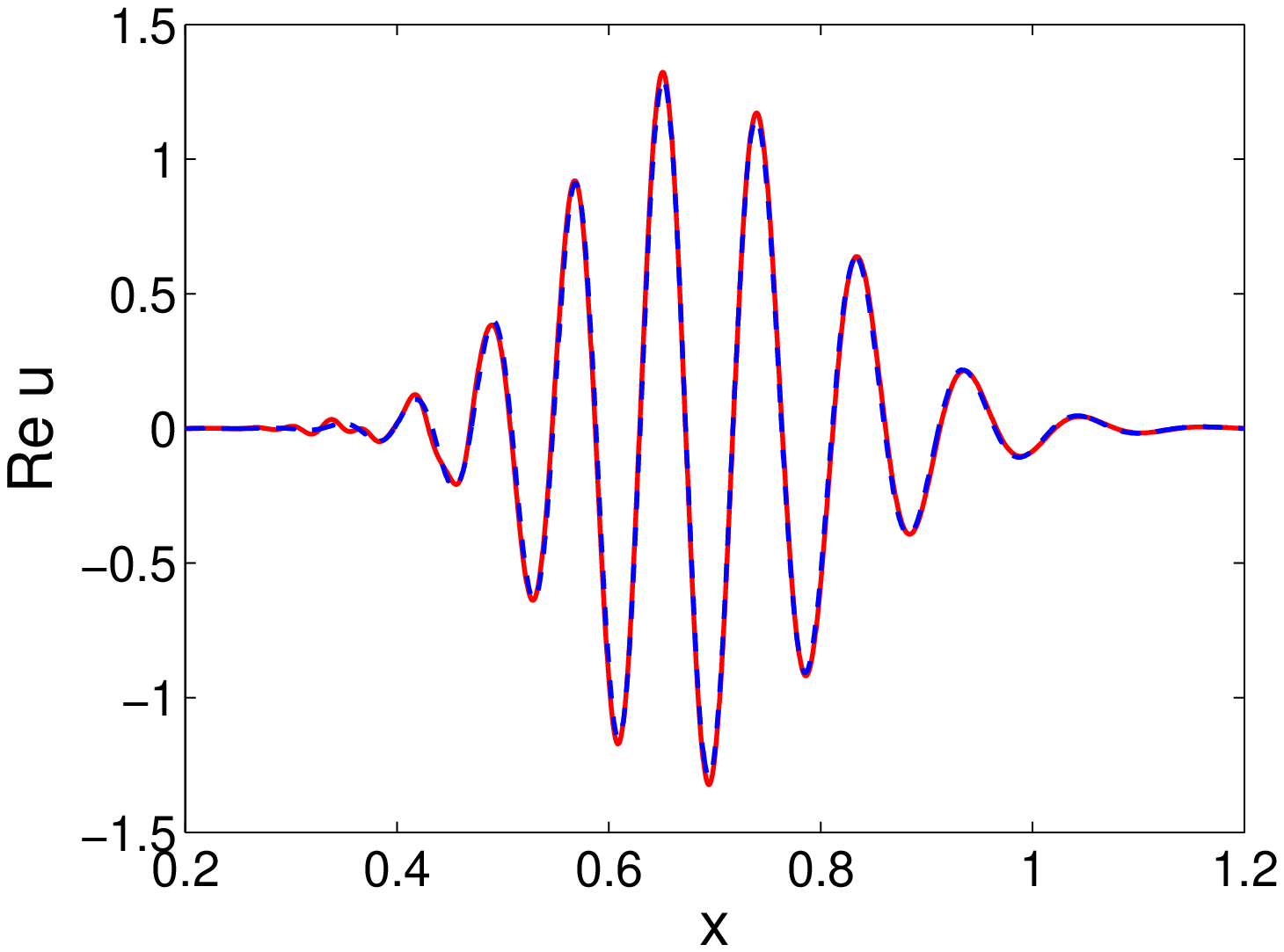}} &
\resizebox{2.3in}{!}{\includegraphics{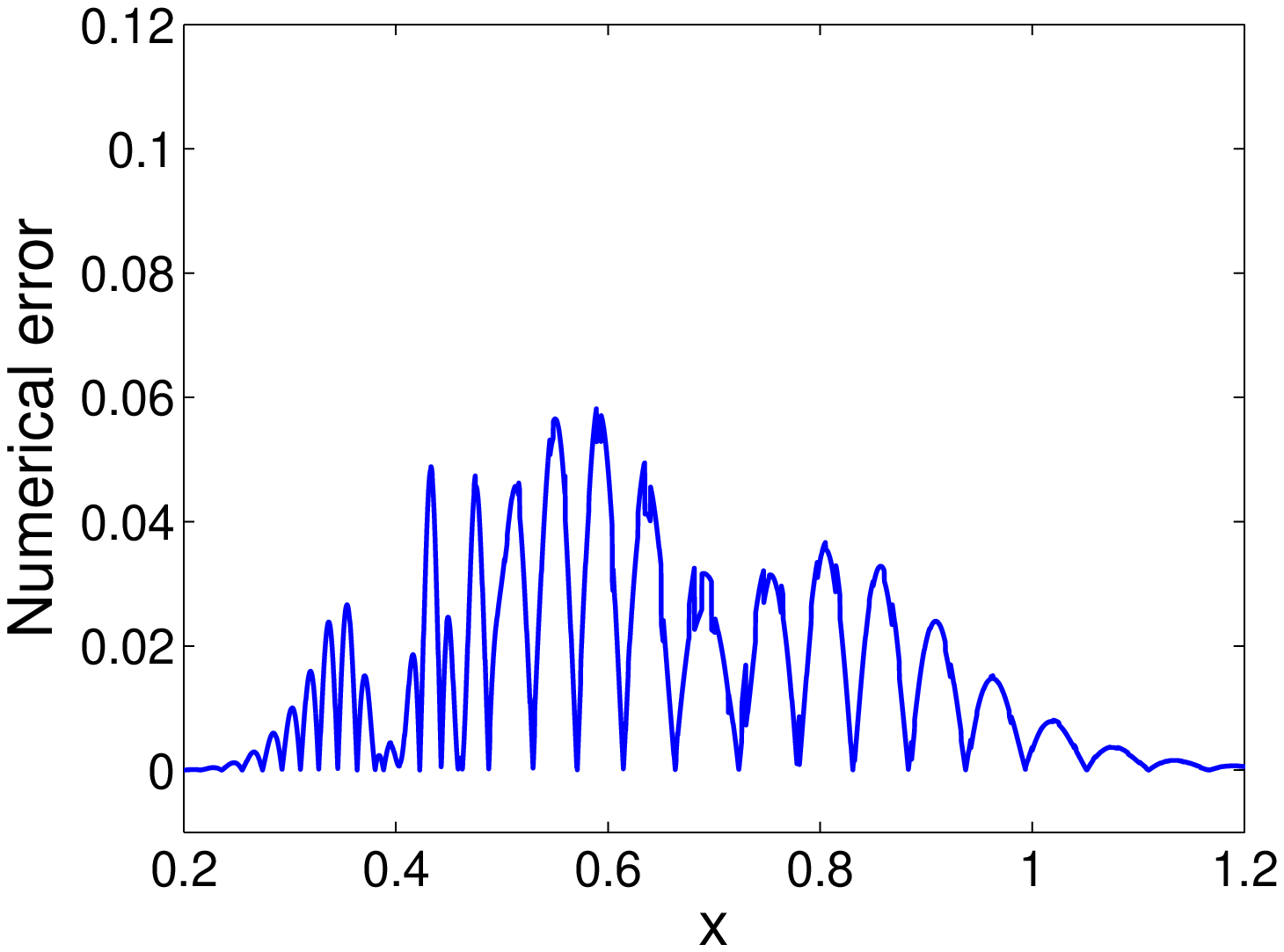}} \\
  \multicolumn{2}{c}{(b) $\veps=\frac{1}{128}$} \\[3mm]
\resizebox{2.3in}{!}{\includegraphics{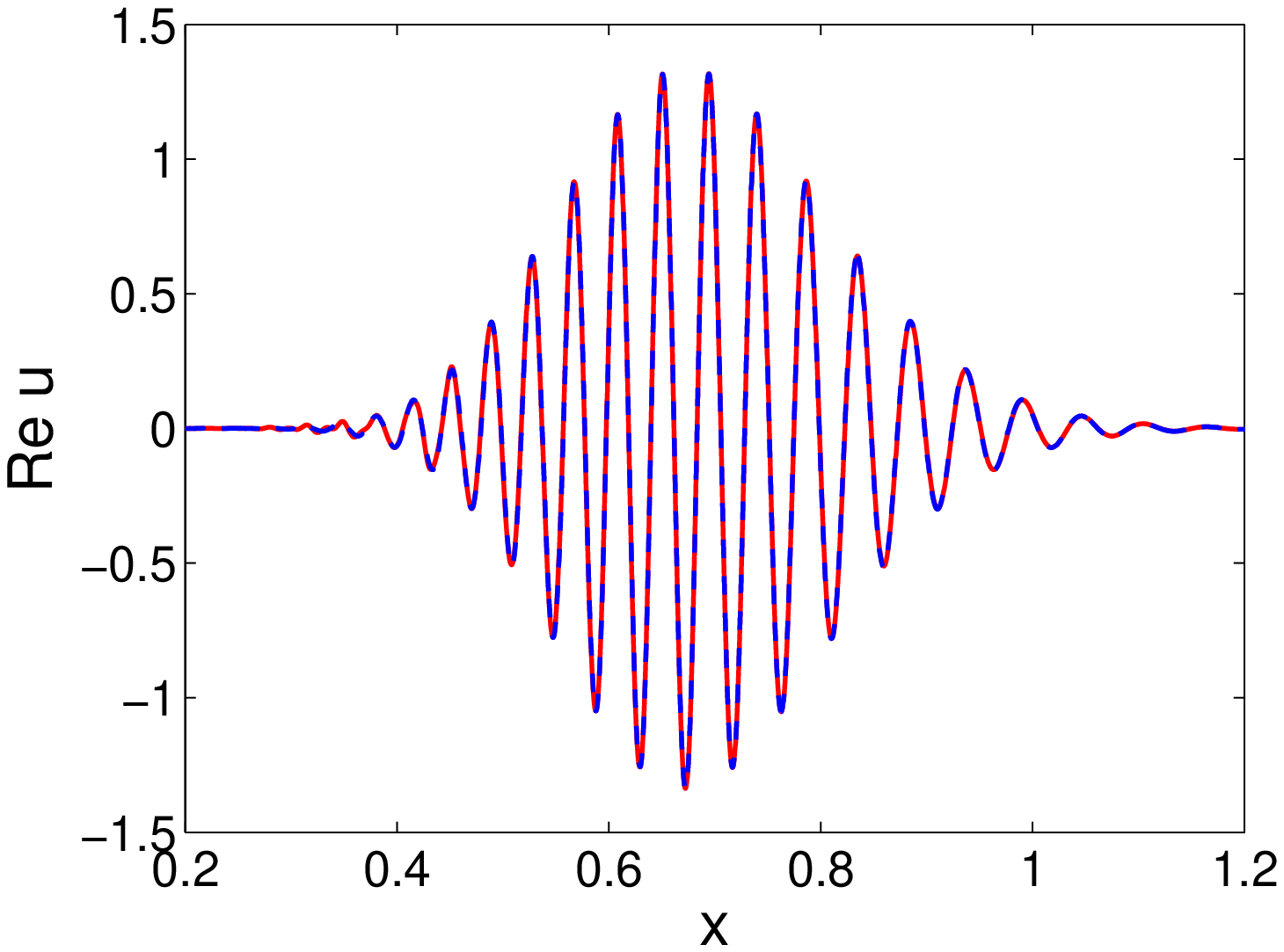}} &
\resizebox{2.3in}{!}{\includegraphics{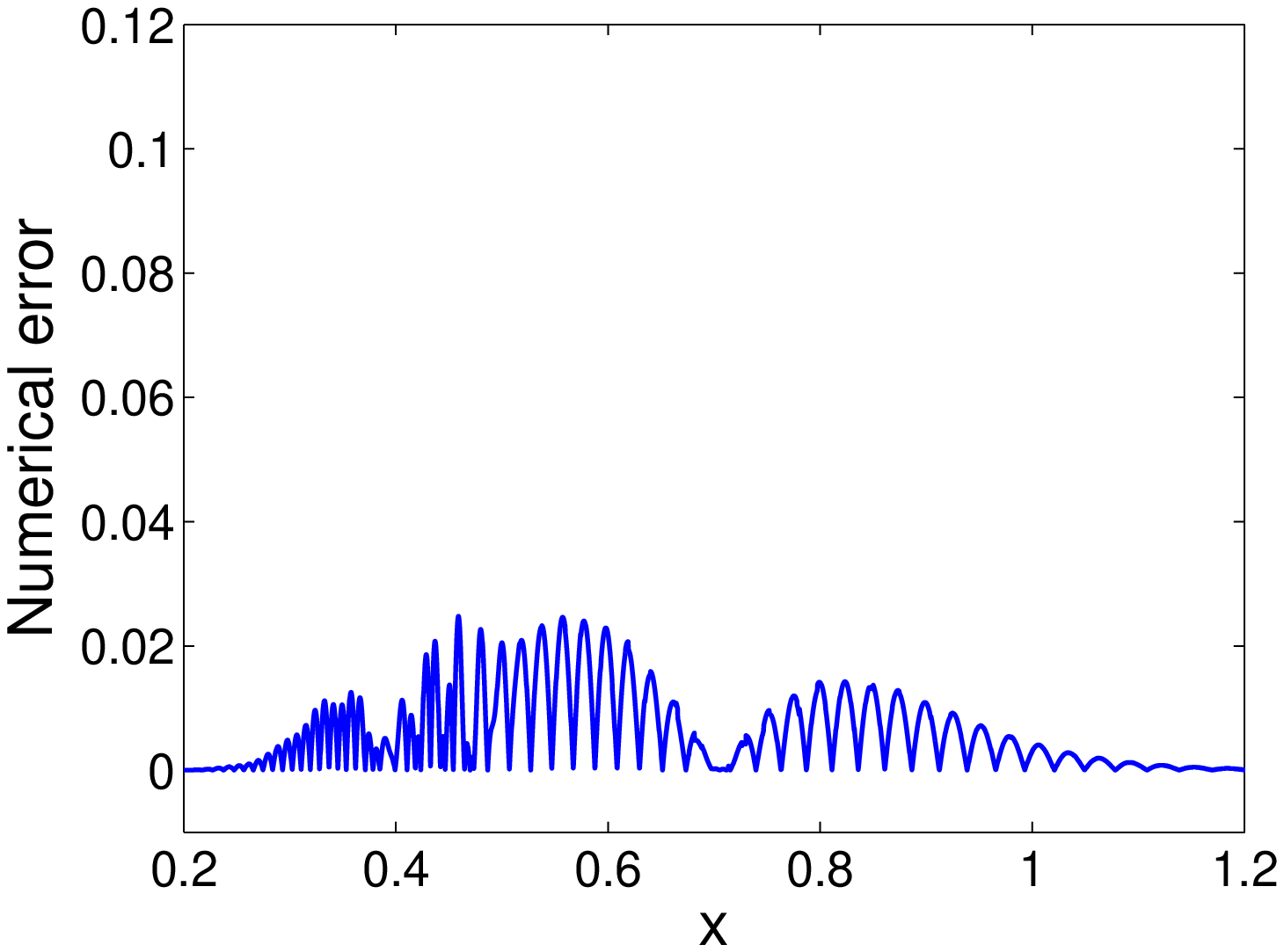}} \\
  \multicolumn{2}{c}{(c) $\veps=\frac{1}{256}$}
\end{tabular}
\caption{Example \ref{exa:1}, the comparison of the true solution
(solid line) and the solution by FGA (dashed line). Left: the real
part of wave field; right: the errors between
them.}\label{fig:ex1_real}
\end{figure}

%\begin{figure}[h t p]
%\begin{tabular}{cc}
%\resizebox{2.3in}{!}{\includegraphics{Ex1/GBHK_real.eps}} &
%\resizebox{2.3in}{!}{\includegraphics{Ex1/GBHK_real_err.eps}}
%\end{tabular}
%\caption{Example \ref{exa:1}, the comparison between the true
%solution (solid line), the solution by FGA
%(dashed line) and the solution by GBM
%(dots) for $\veps=\frac{1}{128}$. Left: the real part of the wave
%field; right: the errors between them (dashed line for FGA, dots for
%GBM).}
%\end{figure}

\begin{example}\label{exa:2}

  The wave speed is $c(x)=x^2$. The initial conditions are
  \begin{align*}
    & u_0=\exp\left(-\frac{(x-0.55)^2}{2\veps}\right)
    \exp\left(\frac{\I x}{\veps} \right),\\
    & \partial_t
    u_0=-\frac{\I x^2}{\veps}\exp\left(-\frac{(x-0.55)^2}{2\veps}\right)
    \exp\left(\frac{\I x}{\veps}\right).
  \end{align*}
\end{example}

We use this example to illustrate the performances of FGA and GBM
when the beams spread in GBM. The final time is $T=1.0$ and
$\veps=1/256$. Remark that the initial condition is chosen as a
single beam on purpose so that one can apply GBM without introducing
any initial errors. The true solution is provided by the
finite difference method using $\delta x=1/2^{11}$ and $\delta
t=1/2^{17}$ for domain $[0,4]$. We take $\delta q=\delta p=\delta
y=1/2^7,\;N_q=128,\;N_p=45$ in FGA to make sure that the error in
the initial value decomposition of FGA is very small. The time step
is $\delta t=1/2^{10}$ for solving the ODEs and the mesh size is
chosen as $\delta x=1/2^{11}$ to construct the final solution in
both FGA and GBM.

Figure \ref{fig:ex2_compare} compares the amplitudes of the wave
field computed by FGA and GBM, and the true solution. One can see
that the beam has spread severely in GBM. The results confirm that
FGA has a good performance even when the beam spreads, while GBM
does not. Moreover, it does not help improving the accuracy if one
uses more Gaussian beams to approximate the initial condition in GBM
as shown in Figure \ref{fig:ex2_multi}, where $N_{y_0}=128$ beams
are used initially and $\delta y_0=1/2^7$. Remark that GBM can still
give good approximation around beam center where Taylor expansion
does not introduce large errors. This can be seen around $x=1.2$ in
Figure \ref{fig:ex2_compare}.

\begin{figure}[h t p]
\begin{tabular}{cc}
  \resizebox{2.3in}{!}{\includegraphics{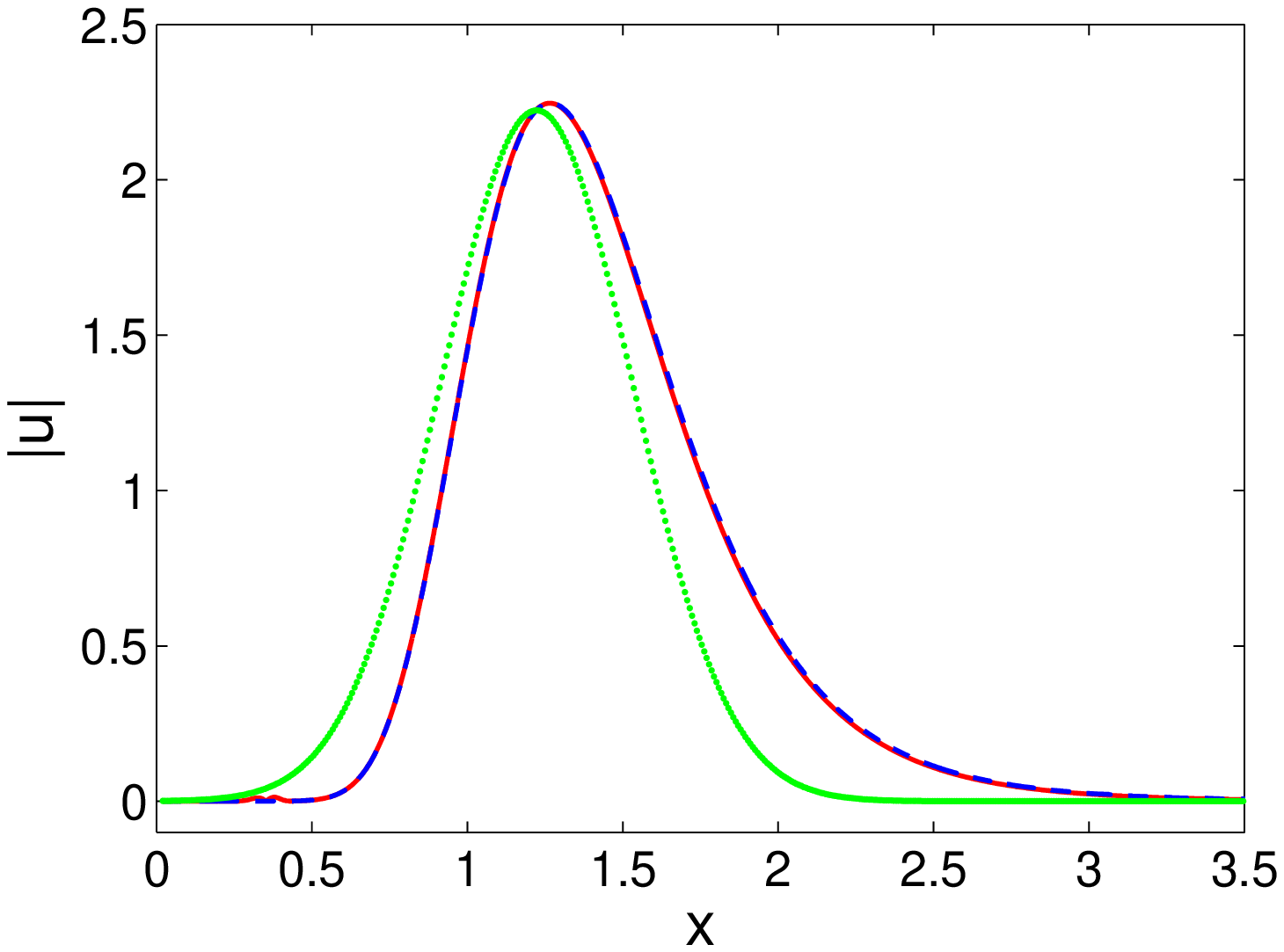}} &
  \resizebox{2.3in}{!}{\includegraphics{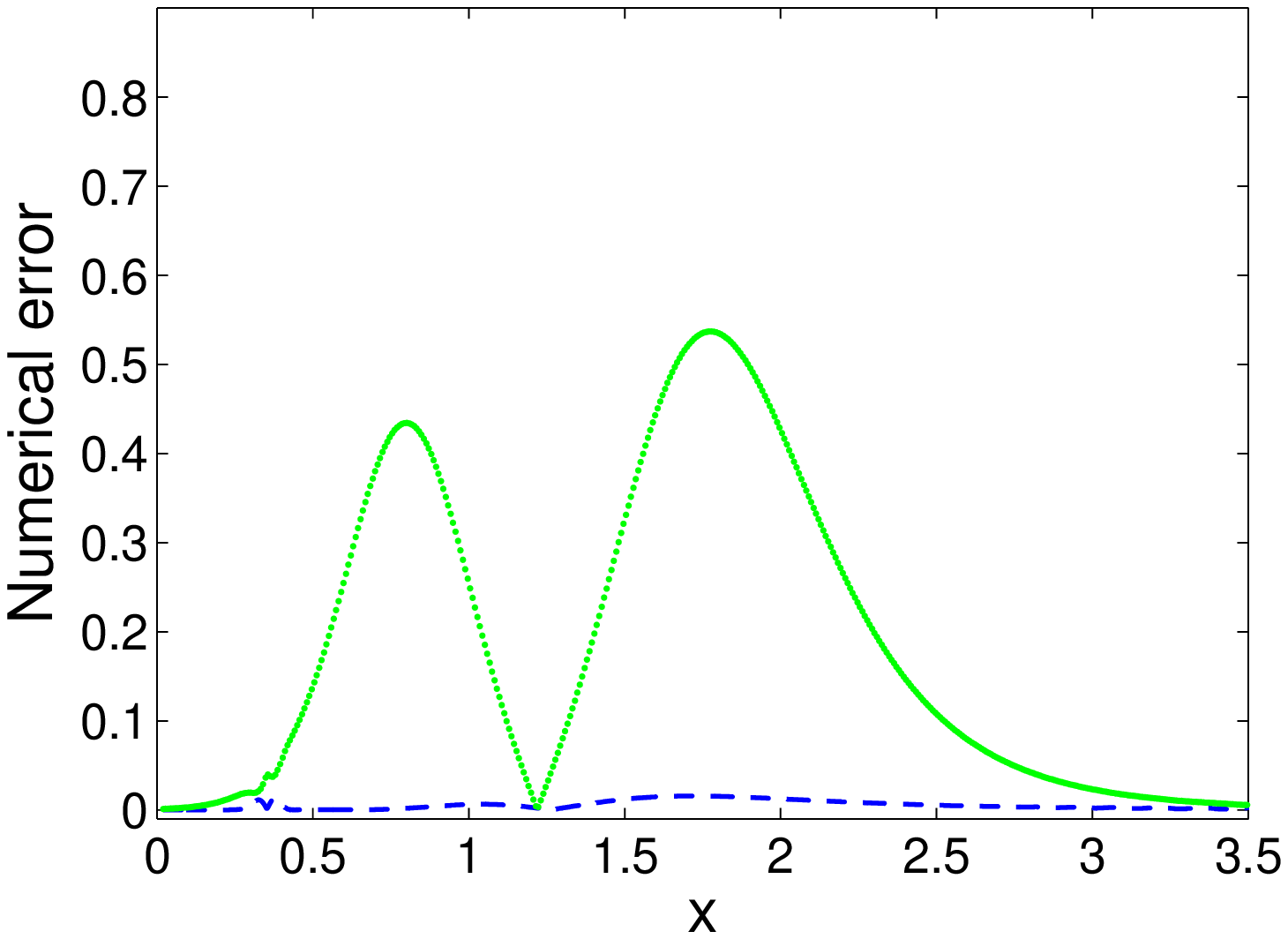}}
\end{tabular}
\caption{Example \ref{exa:2}, the comparison of the true solution
  (solid line), the solution by FGA (dashed line) and the
  solution by GBM (dots) for $\veps=\frac{1}{256}$.  Left:
  the amplitude of wave field; right: the error between them (dashed
  line for FGA, dots for GBM).}\label{fig:ex2_compare}
\end{figure}

\begin{figure}[h t p]
\begin{tabular}{cc}
\resizebox{2.3in}{!}{\includegraphics{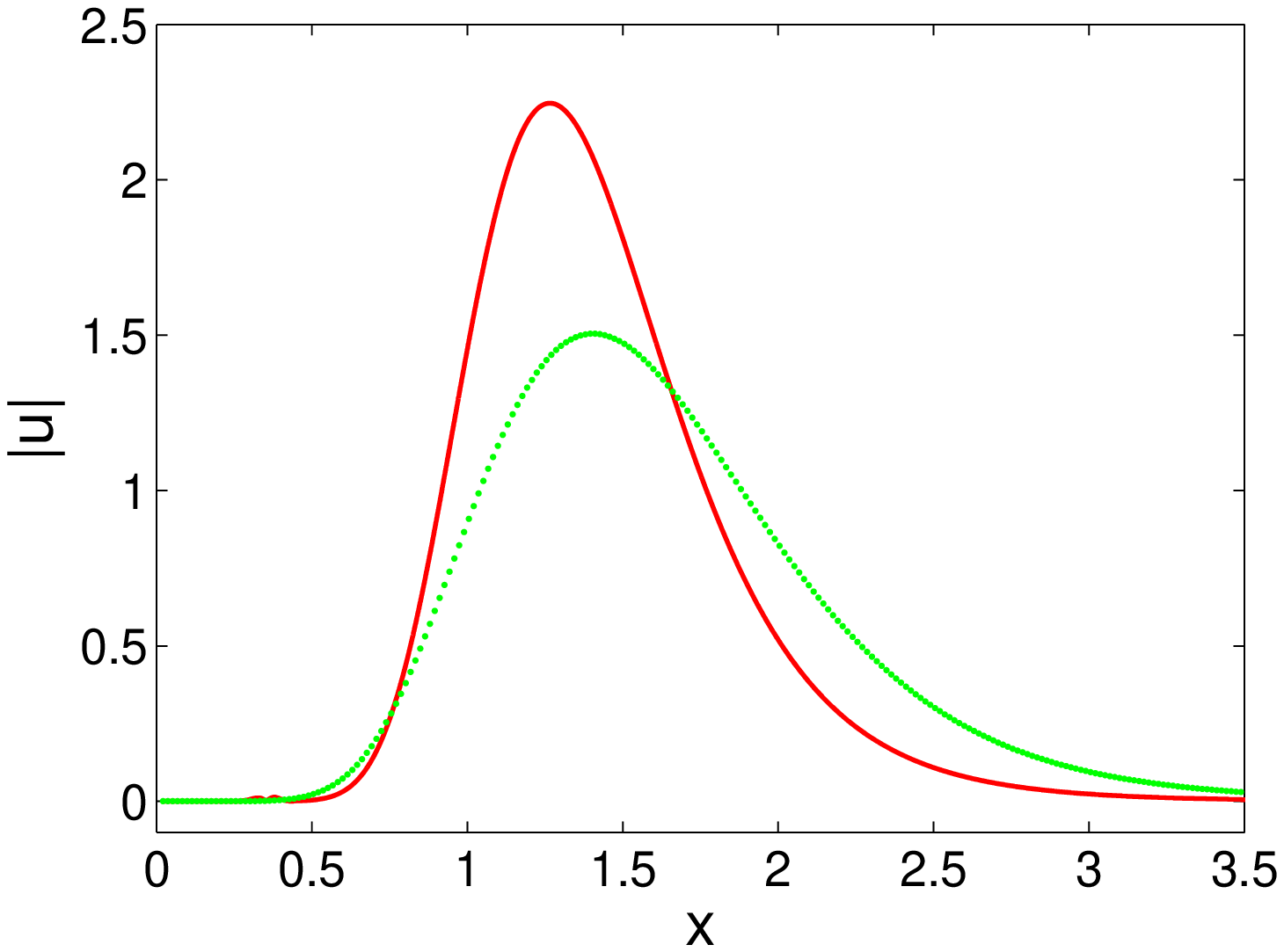}} &
\resizebox{2.3in}{!}{\includegraphics{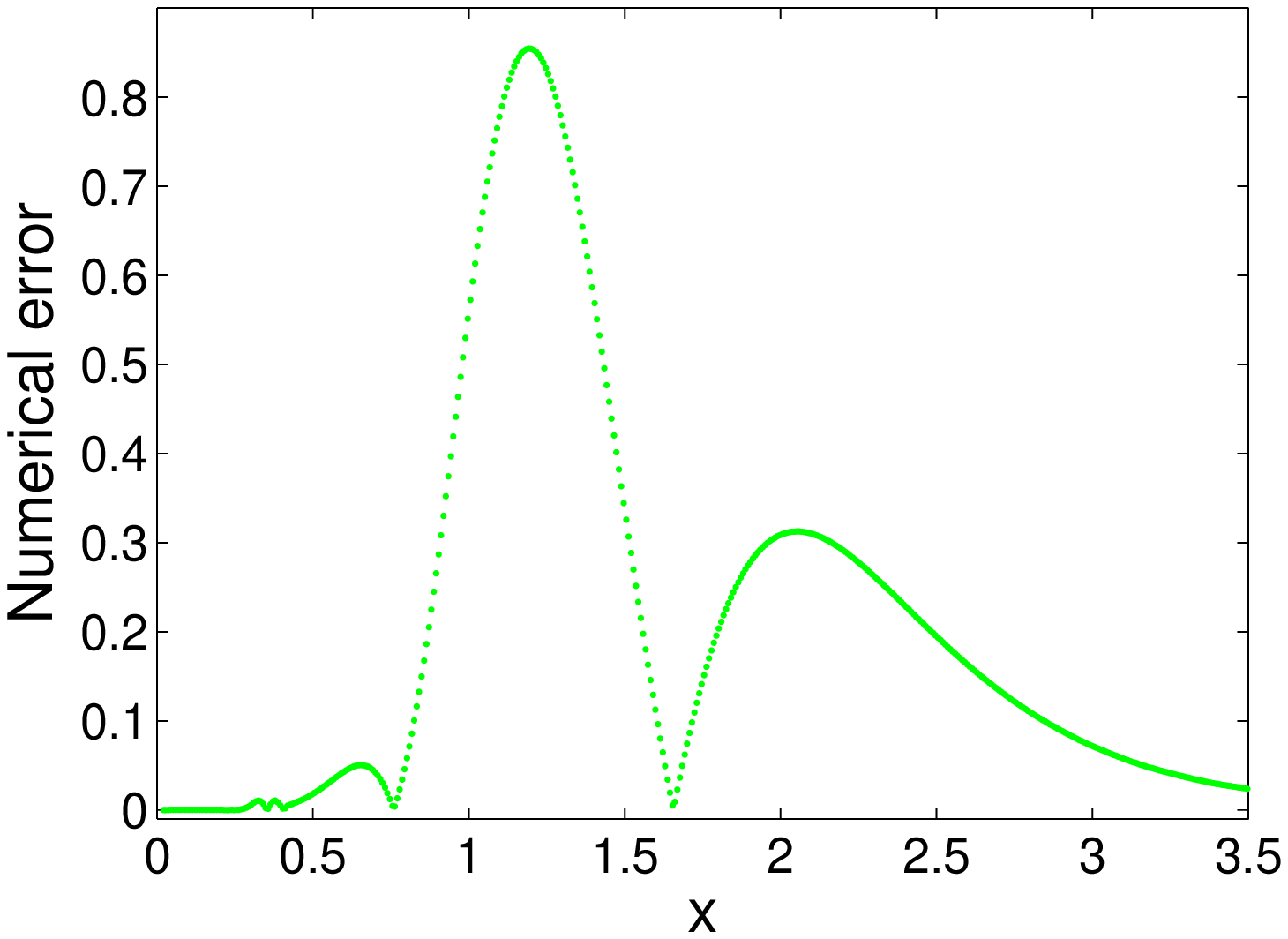}}
\end{tabular}
\caption{Example \ref{exa:2}, the comparison of the true solution
  (solid line) and the solution by GBM using multiple
  Gaussian initial representation (dots) for
  $\veps=\frac{1}{256}$. Left: the amplitude of wave field; right:
  the error between them.}\label{fig:ex2_multi}
\end{figure}

\subsection{Two dimension}

\begin{example}\label{exa:3}

The wave speed is $c(x_1,x_2)=1$. The initial conditions are
\begin{align*}
  & u_0=\exp\bigl(-100(x_1^2+x_2^2)\bigr)
  \exp\left(\frac{\I}{\veps}(-x_1+\cos(2x_2))\right),\\
  & \partial_t u_0=-\frac{\I}{\veps}\sqrt{1+4\sin^2(2x_2)}
  \exp\bigl(-100(x_1^2+x_2^2)\bigr)
  \exp\left(\frac{\I}{\veps}(-x_1+\cos(2x_2))\right).
\end{align*}
\end{example}

This example presents the cusp caustics shown in Figure
\ref{fig:ex3_caustic}. The final time is $T=1.0$. The true solution
is given by the spectral method using the mesh $\delta x_1=\delta
x_2=1/512$ for domain $[-1.5,0.5]\times[-1,1]$. We take $\delta
q_1=\delta q_2=\delta p_1=\delta p_2=\delta y_1=\delta y_2=1/32$,
$N_{q}=32,\;N_{p}=8$ in FGA, and use $\delta x_1=\delta x_2=1/128$
to reconstruct the solution. Figure \ref{fig:ex3_amp} compares the
wave amplitude of the true solution and the one by FGA for
$\veps=1/128$ and $1/256$. The $\ell^\infty$ and $\ell^2$ errors of
the wave amplitude are $1.98\times10^{-1}$ and $4.42\times10^{-2}$
for $\veps=1/128$, and $1.07\times10^{-1}$ and $2.20\times10^{-2}$
for $\veps=1/256$. This shows a linear convergence in $\veps$ of the
method.

% Note that in this example the mesh sizes are not small and
% $N_p$ is not large enough to make the asymptotic error dominate. The
% convergence rate comes from the fact that the approximation of
% initial conditions under such meshes converges linearly in $\veps$.
% We comment that this is different from Example \ref{exa:1} where the
% mesh sizes and $N_p$ are taken fine enough so that the numerical
% error comes mostly from asymptotic approximation.

\begin{figure}[h t p]
\begin{tabular}{cc}
\resizebox{2.3in}{!}{\includegraphics{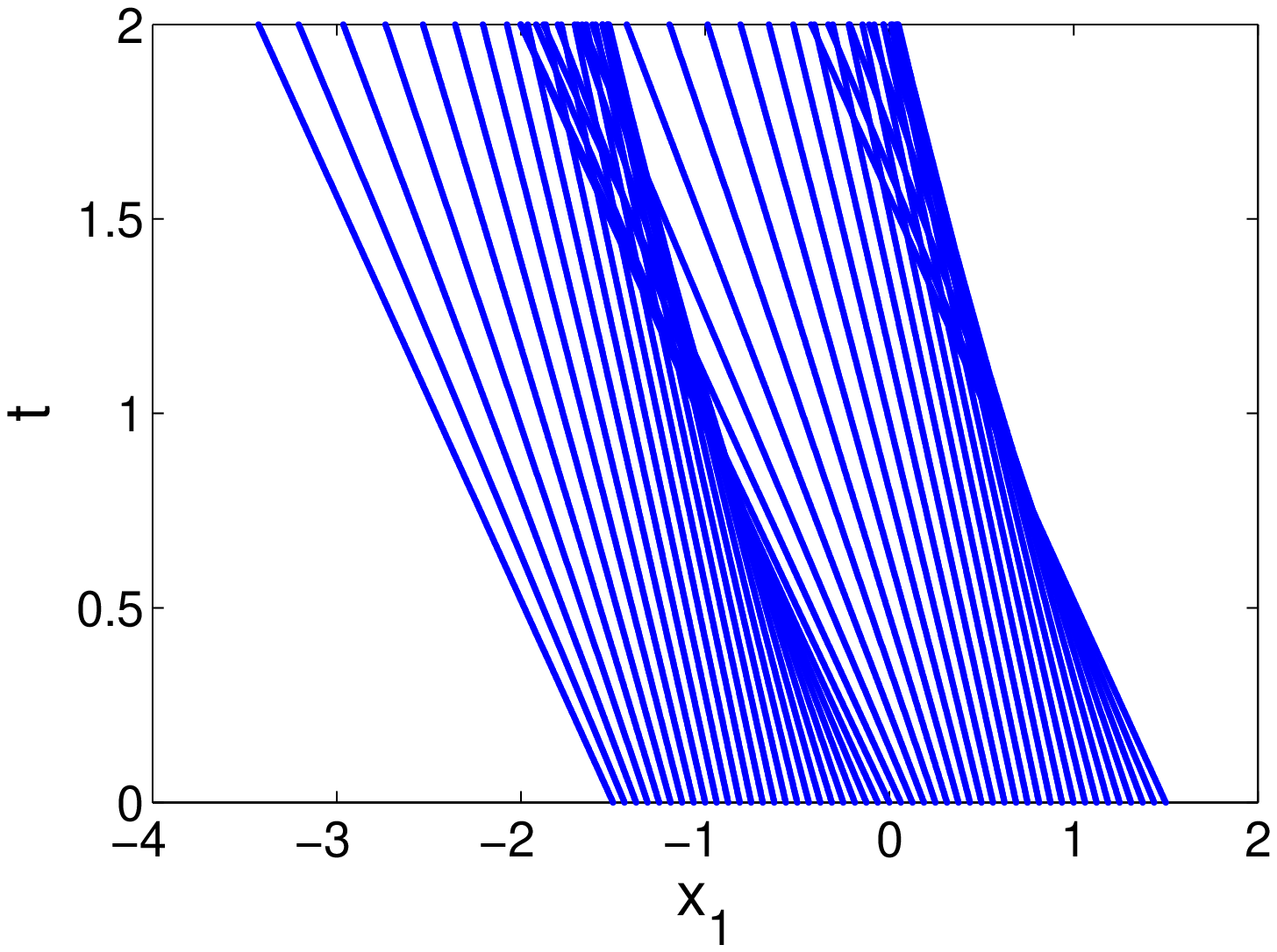}} &
\resizebox{2.3in}{!}{\includegraphics{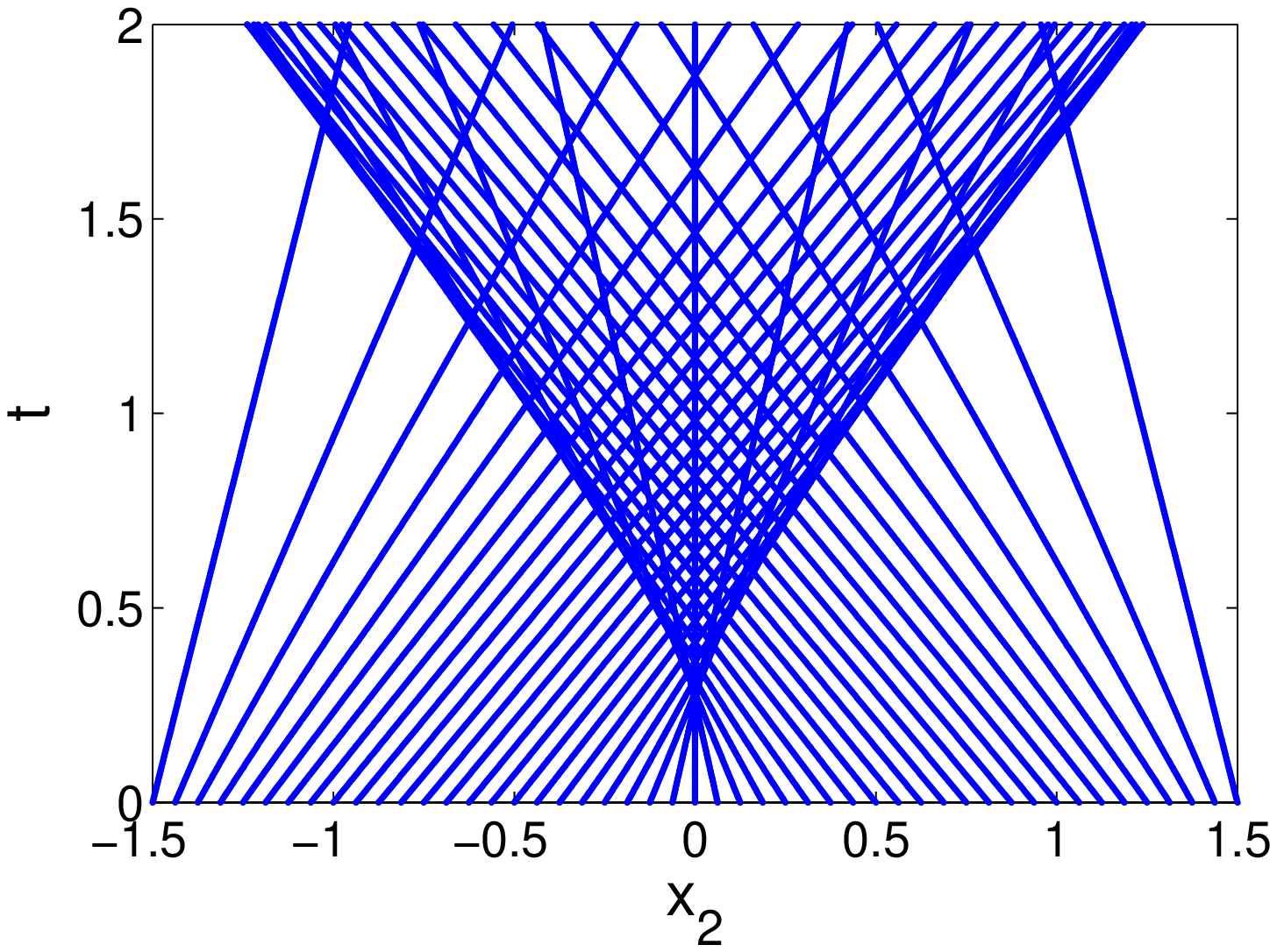}}
\end{tabular}
\caption{Example \ref{exa:3}, a set of the characteristic lines
develops the cusp caustic.}\label{fig:ex3_caustic}
\end{figure}

\begin{figure}[h t p]
\begin{tabular}{cc}
  \resizebox{2.3in}{!}{\includegraphics{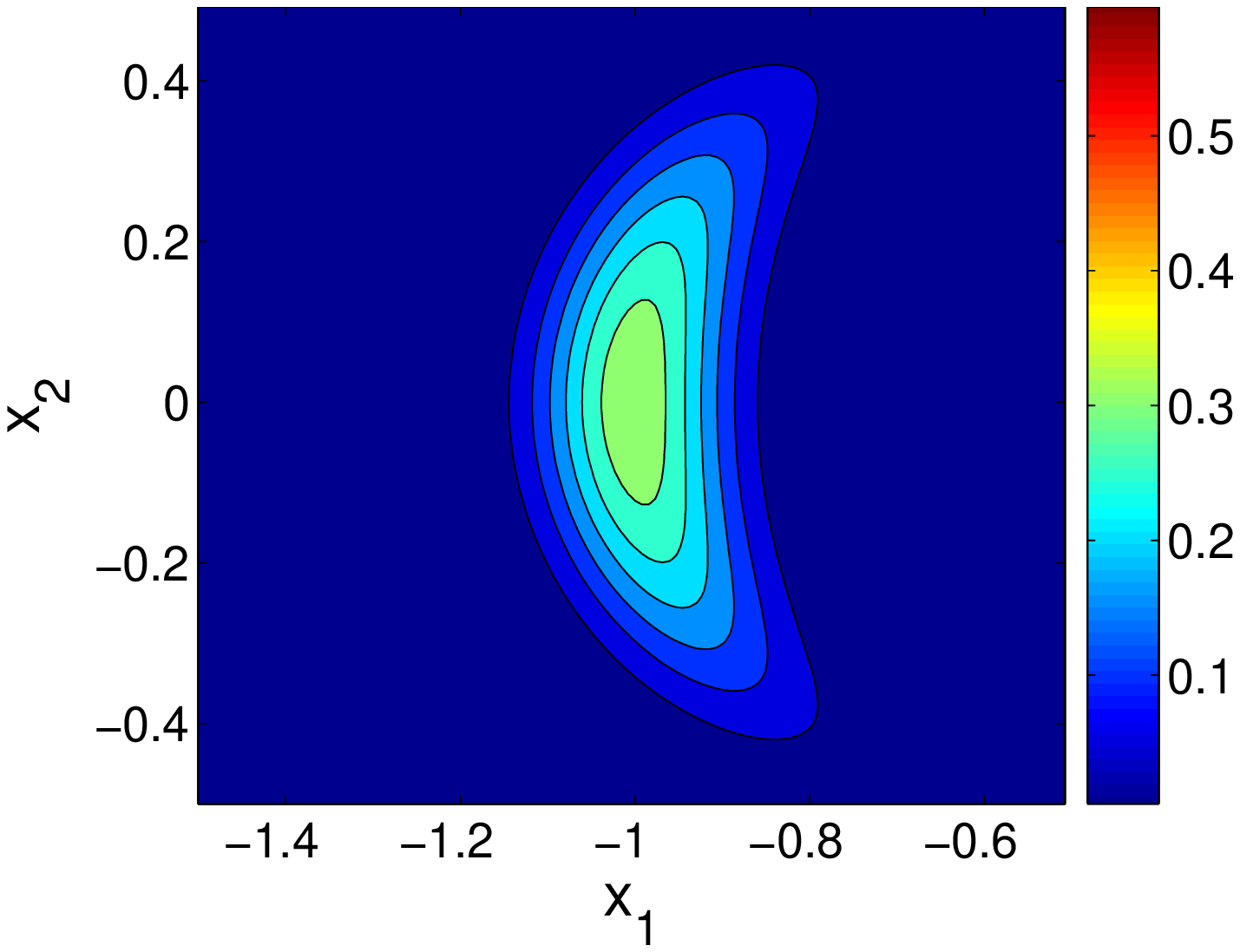}} &
  \resizebox{2.3in}{!}{\includegraphics{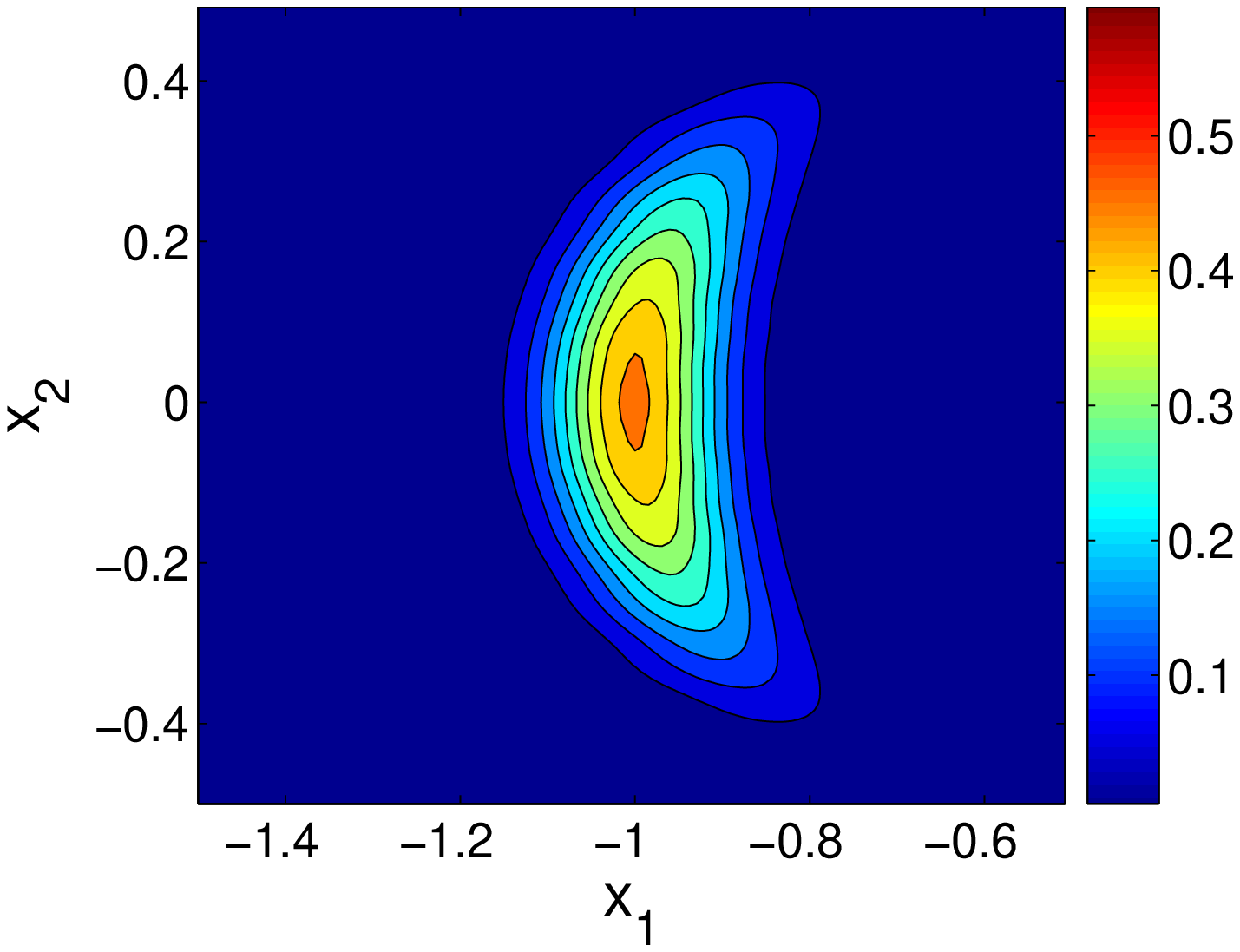}} \\
  \multicolumn{2}{c}{(a) Frozen Gaussian approximation} \\[3mm]
  \resizebox{2.3in}{!}{\includegraphics{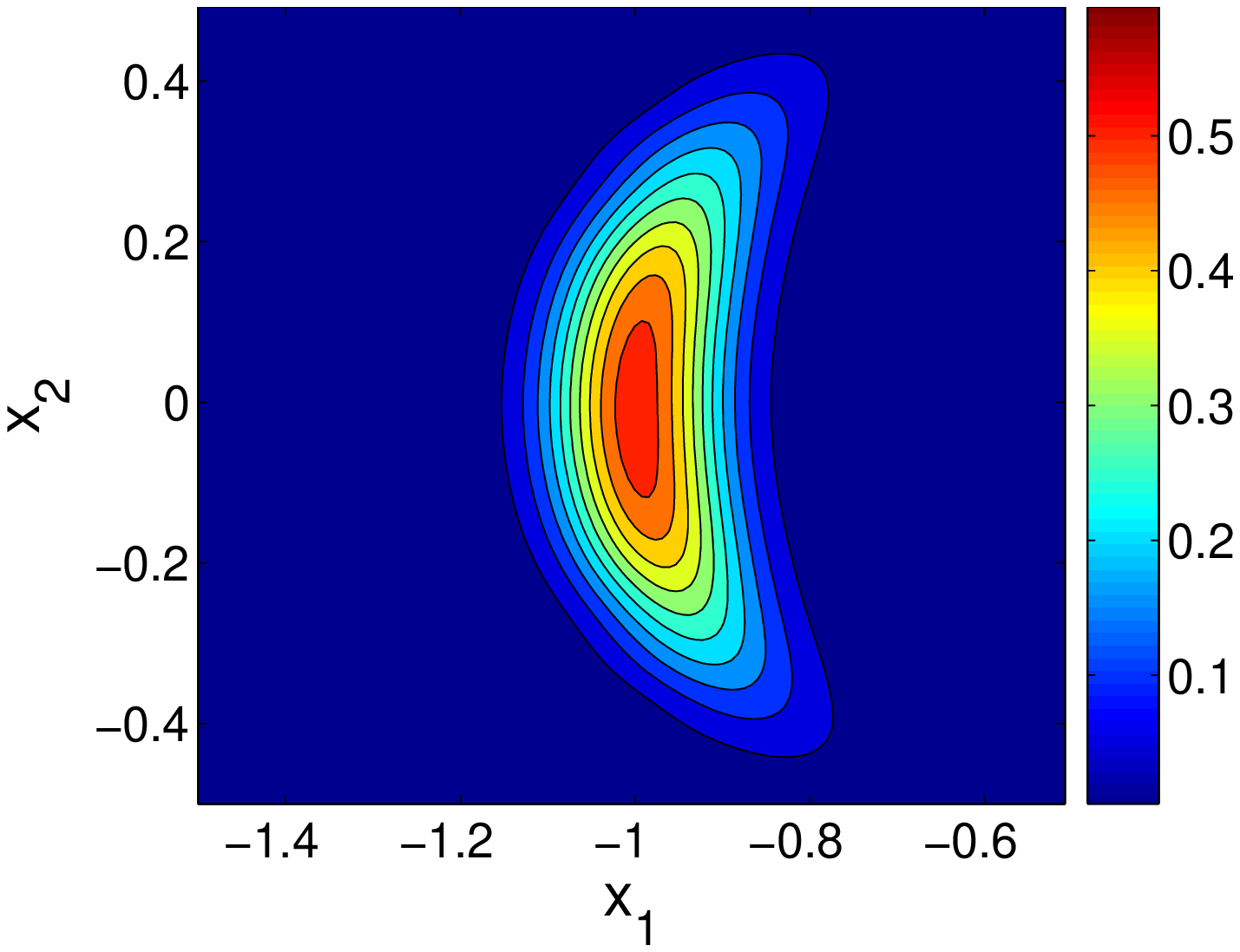}} &
  \resizebox{2.3in}{!}{\includegraphics{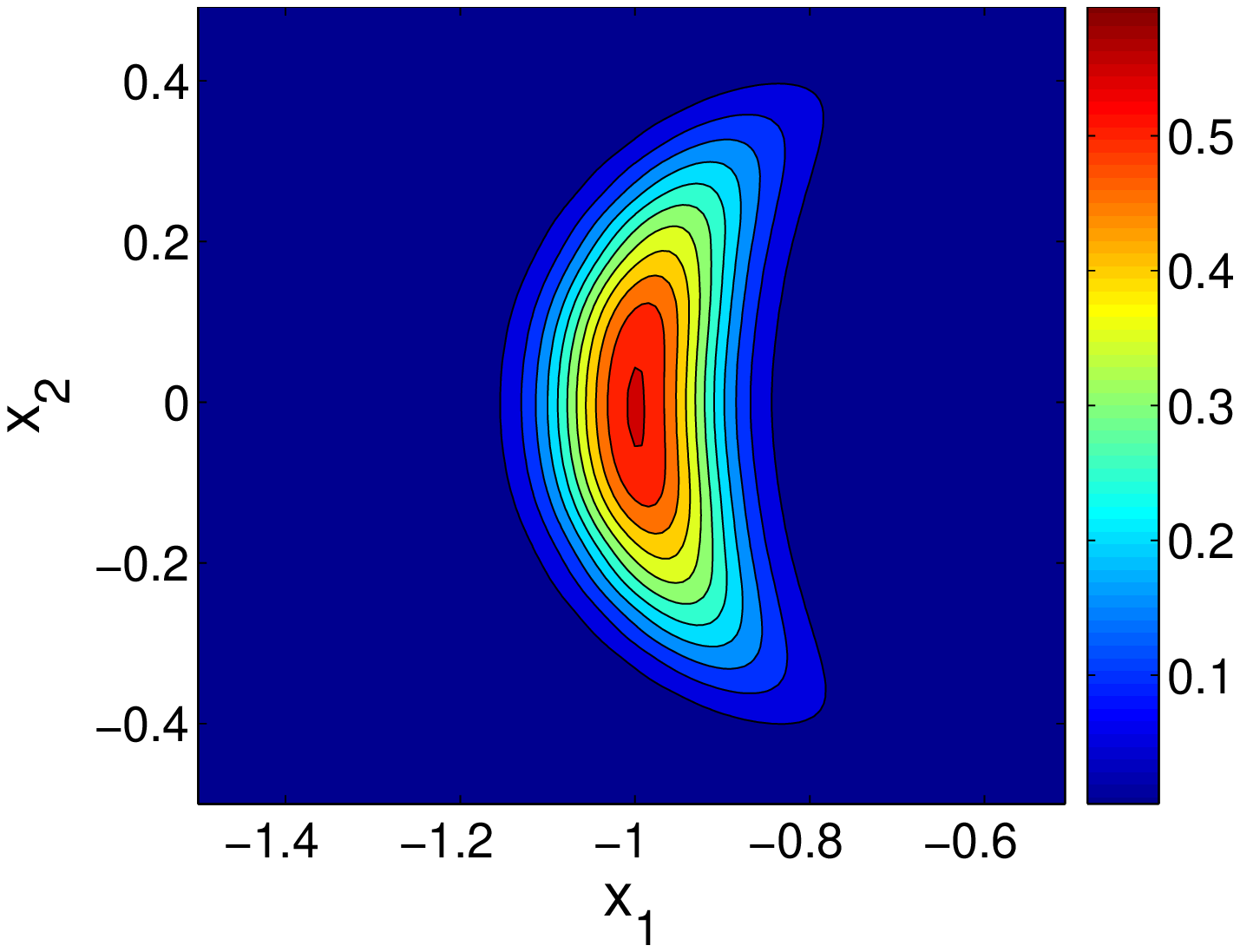}} \\
  \multicolumn{2}{c}{(b) True solution} \\[3mm]
  \resizebox{2.3in}{!}{\includegraphics{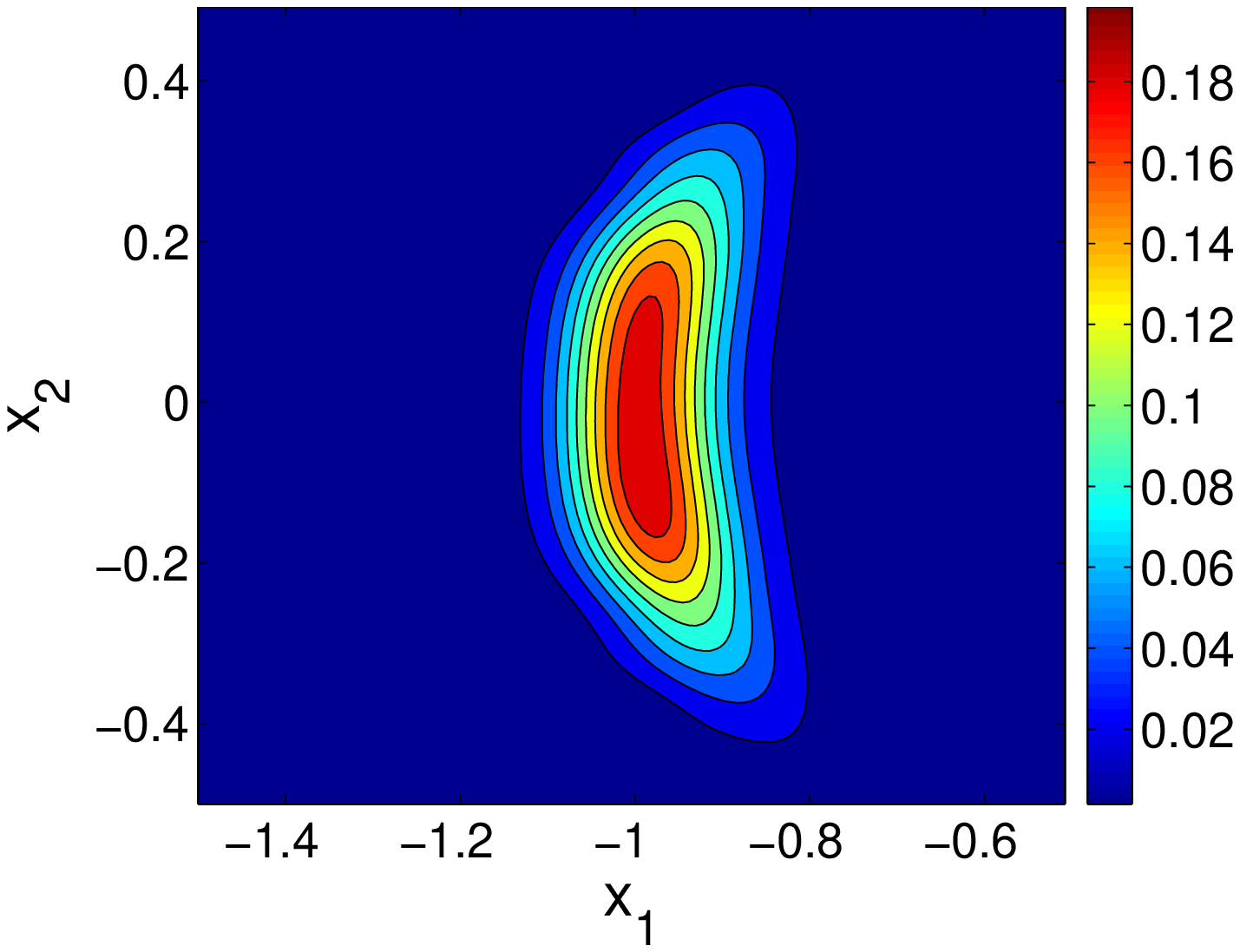}} &
  \resizebox{2.3in}{!}{\includegraphics{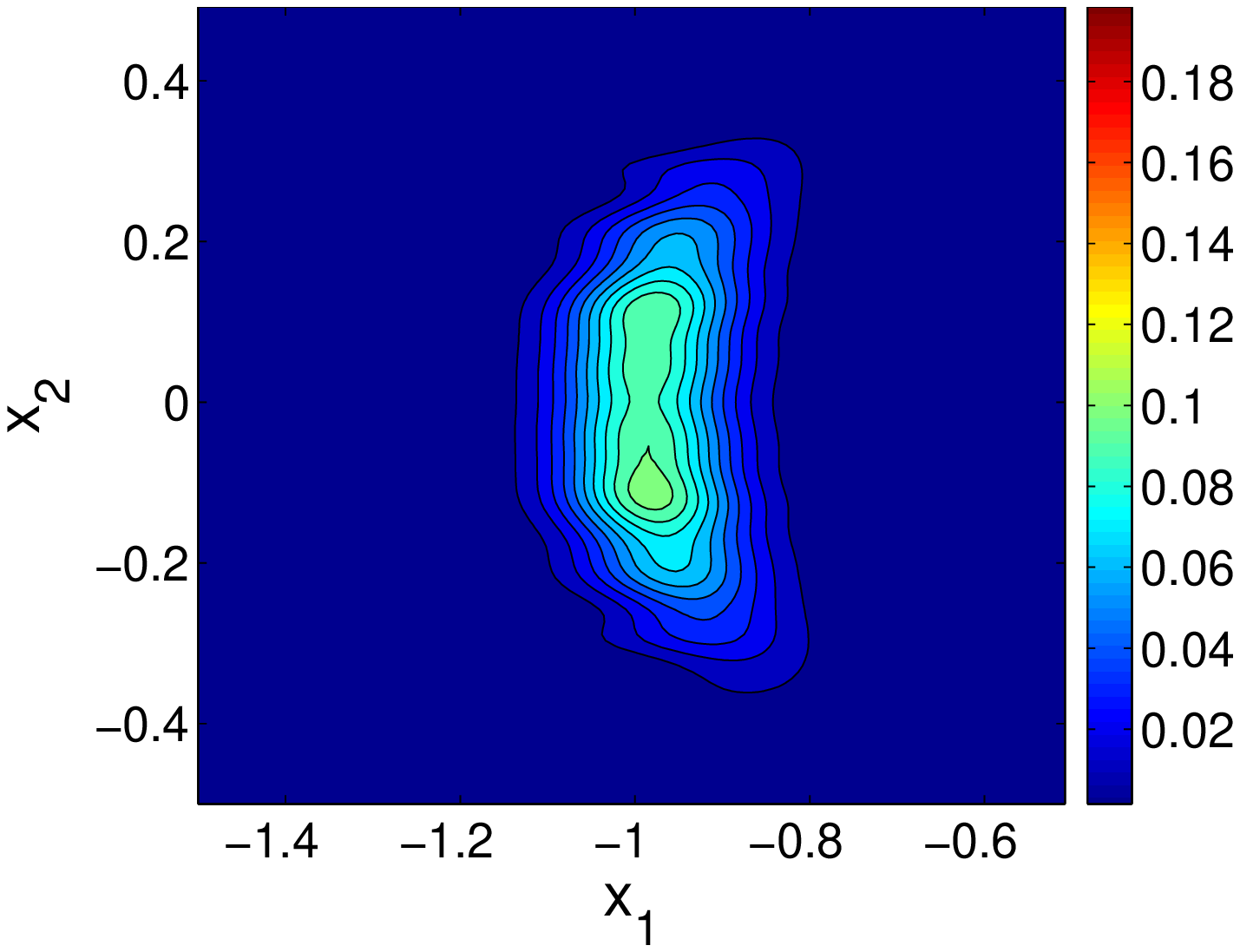}} \\
  \multicolumn{2}{c}{(c) Errors}
\end{tabular}
\caption{Example \ref{exa:3}, the comparison of the true solution
and the solution by FGA. Left: wave amplitude of
$\veps=\frac{1}{128}$;
  right: wave amplitude of $\veps=\frac{1}{256}$.}\label{fig:ex3_amp}
\end{figure}

\section{Discussion and Conclusion}\label{sec:conclusion}

We first briefly compare the efficiency of frozen Gaussian
approximation (FGA) with the Gaussian beam method (GBM). GBM uses only
one Gaussian function for each grid point in physical space, while FGA
requires more Gaussians per grid point with different initial momentum
to capture the behavior of focusing or spreading of the
solution. However, the stationary phase approximation suggests that
the number of Gaussians is only increased by a small constant multiple
of the number of those used in GBM. In addition, in GBM one has to
solve the Riccati equation, which is a coupled nonlinear ODE system in
high dimension, to get the dynamics of the Hessian matrix for each
Gaussian, while in FGA the Hessian matrix is determined initially and
has no dynamics.  Therefore, the overall efficiency of FGA is
comparable to GBM.

Admittedly, higher order GBM gives better asymptotic accuracy, and
only requires solving a constant number of additional ODEs as in FGA.
The numerical cost of higher order GBM is comparable to FGA. However, 
higher order GBM has its drawbacks: The imaginary part of higher order 
(larger than two) tensor function dose not preserve 
positive definiteness in time evolution, which may destroy the decay property of 
the ansatz of higher order GBM . This is even more severe when beams
spread. Moreover, the ODEs in higher order GBM are in the form of
coupled nonlinear system in high dimension.  It raises numerical
difficulty caused by stability issues. We also note that, the
numerical integration of ODEs in FGA can be easily parallelized since
the Hamiltonian flow \eqref{eq:characline} is independent for
different initial $(\bd{q},\bd{p})$, while it is not so trivial for
higher order tensors in GBM.

From the accuracy point of view, our numerical examples show that
first order FGA method has asymptotic accuracy $\Or(\veps)$. The
existing rigorous analysis (\cites{CoRo:97, BoAkAl:09, LiRuTa:10})
proves that the $k$-th order GBM has an accuracy of
$\Or(\veps^{k/2})$. Hence, at the first order, FGA has better
asymptotic accuracy than GBM. We note that, however, there has been
numerical evidence presenting $\Or(\veps)$ asymptotic accuracy order
for first order GBM, for example in
\cites{JiWuYa:08,JiWuYa:11,MoRu:app,LiRuTa:10}. This phenomenon is
usually attributed to error cancellation between different beams. To
the best of our knowledge, the mechanism of error cancellation in
GBM has not been systematically understood yet. 

With the the gain of halfth order in asymptotic accuracy due to
cancellation, the first order GBM has the same accuracy order as FGA
(of course GBM still loses accuracy when beams spread). Remark that
the gain in asymptotic accuracy order depends on the choice of
norm. For example, the first order GBM has a halfth order convergence
in $\ell^\infty$ norm, first order convergence in $\ell^2$ norm and
$3/2$-th order convergence in $\ell^1$ norm in Example $1$ of
\cite{JiWuYa:08}. Moreover, the error cancellation seems not to be
easily observed in numerics unless $\veps$ is very small. For
instance, the convergence order of GBM in Example \ref{exa:1} is only
a bit better than $1/2$ for $\veps$ up to $1/256$. While in FGA, we
numerically observe the first order asymptotic accuracy in both
$\ell^2$ and $\ell^{\infty}$ norms.

Actually the accuracy of FGA can also be understood from
a viewpoint of error cancellation. Note that the equalities \eqref{eq:con1}, 
\eqref{eq:con2} and \eqref{eq:con3} in Lemma \ref{lem:veps1} play
the role of determining the accuracy of FGA. In \eqref{eq:con1}, the
term $\bd{x}-\bd{Q}$ is of order $\Or(\sqrt{\veps})$ due to the
Gaussian factor, but after integration with respect to $\bd{q}$ and
$\bd{p}$, which is similar to the beam summation in GBM, it becomes
$\Or(\veps)$. Similar improvement of order also happens in
\eqref{eq:con2} and \eqref{eq:con3}. Integration by parts along with
\eqref{eq:dzPhi} explains the mechanism of this type of error
cancellation.

We conclude the paper as follows. In this work, we propose the
frozen Gaussian approximation (FGA) for computation of high
frequency wave propagation, motivated by the Herman-Kluk propagator
in chemistry literature. This method is based on asymptotic analysis
and constructs the solution using Gaussian functions with fixed
widths that live on the phase plane. It not only provides an
accurate asymptotic solution in the presence of caustics, but also
resolves the problem in the Gaussian beam method (GBM) when beams
spread. These merits are justified by numerical examples.
Additionally, numerical examples also show that FGA exhibits better
asymptotic accuracy than GBM. These advantages make FGA quite
competitive for computing high frequency wave propagation.

For the purpose of presenting the idea simply and clearly, we only
describe the method for the linear scalar wave equation using leading
order approximation. The method can be generalized for solving other
hyperbolic equations and systems with a character of high
frequency. The higher order approximation can also be derived.  Since
the method is of Lagrangian type, the issue of divergence still
remains, which will be resolved in an Eulerian framework. We present
these results in the subsequent paper \cite{LuYang:MMS}.

\FloatBarrier

\bibliographystyle{amsalpha}
\bibliography{hermankluk}

\end{document}